\documentclass{amsart}
\usepackage{latexsym,amsxtra,amscd,ifthen}
\usepackage{graphicx}
\usepackage{amsfonts}
\usepackage{verbatim}
\usepackage{amsmath}
\usepackage{amsthm}
\usepackage{amssymb}
\usepackage{url}
\usepackage{color}

\theoremstyle{plain}

\newtheorem{theorem}{Theorem}
\newtheorem{lemma}[theorem]{Lemma}
\newtheorem{proposition}[theorem]{Proposition}
\newtheorem{hypothesis}[theorem]{Hypothesis}
\newtheorem{corollary}[theorem]{Corollary}

\numberwithin{theorem}{section}
\numberwithin{equation}{theorem}

\theoremstyle{definition}
\newtheorem{definition}[theorem]{Definition}
\newtheorem{example}[theorem]{Example}
\newtheorem{remark}[theorem]{Remark}
\newtheorem{question}[theorem]{Question}
\newtheorem*{question*}{Question}

\DeclareMathOperator{\ldt}{ldt}
\DeclareMathOperator{\hdet}{hdet}

\DeclareMathOperator{\Ext}{Ext}
\DeclareMathOperator{\Hom}{Hom}
\DeclareMathOperator{\Mod}{Mod}

\DeclareMathOperator{\Aut}{Aut}
\DeclareMathOperator{\gr}{gr}

\DeclareMathOperator{\ML}{ML}
\DeclareMathOperator{\LND}{LND}

\begin{document}
\title{Nakayama automorphism and applications}

\author{J.-F. L\"u}
\address{(L\"u) Department of Mathematics,
Zhejiang Normal University, Jinhua 321004, China}
\email{jiafenglv@zjnu.edu.cn}

\author{X.-F. Mao}
\address{(Mao) Department of Mathematics,
Shanghai University, Shanghai 200444, China}
\email{xuefengmao@shu.edu.cn}

\author{J.J. Zhang}
\address{(Zhang) Department of Mathematics,
University of Washington, Seattle 98105, USA}
\email{zhang@washington.edu}

\begin{abstract}
Nakayama automorphism is used to study group actions and Hopf 
algebra actions on Artin-Schelter regular algebras of global 
dimension three.
\end{abstract}




\maketitle


\setcounter{section}{-1}

\section{Introduction}
\label{xxsec0}
Throughout let $k$ be an algebraically closed field of characteristic zero.
All algebras and Hopf algebras are over $k$.

The main motivation comes from noncommutative invariant theory. In a recent
paper \cite{CKWZ}, Chan-Kirkman-Walton-Zhang classified and studied
finite dimensional Hopf algebra actions on Artin-Schelter regular algebras of global dimension two with trivial homological determinant.
One ultimate goal is to carry out the same project for
Artin-Schelter regular algebras of  global dimension three:

\begin{question}\cite[Question 8.1]{CKWZ}.
\label{xxque0.1}
Let $A$ be an Artin-Schelter regular algebra of global dimension 3. 
What are the finite dimensional Hopf algebras that act inner faithfully 
on and preserve the grading of $A$ with trivial homological determinant?
\end{question}

Keeping in mind this ultimate goal, let us ask some interesting and
closely related questions that make sense even in higher dimensional cases.

\begin{question}
\label{xxque0.2}
Let $A$ be  an Artin-Schelter regular algebra
and let $H$ be a finite dimensional  Hopf algebra acting on  $A$
inner faithfully.
\begin{enumerate}
\item[(1)]
Under what hypotheses on $A$, must $H$ be semisimple?
\item[(2)]
Under what hypotheses on $A$, must $H$ be a group algebra?
\item[(3)]
Under what hypotheses on $A$, must $H$ be the dual of a group algebra
(namely, its dual Hopf algebra $H^\circ$ is a group algebra)?
\item[(4)]
Assuming $H$ is semisimple, under what hypotheses on $A$, must $H$ be
a group algebra (respectively, the dual of a group algebra)?
\item[(5)]
Under what hypotheses on $A$, does every $H$-action preserve the 
grading of $A$? Assume, further, that $H$ is a group algebra 
(respectively, the dual of a group algebra). Under what hypotheses
on $A$, does the $H$-action preserve the grading of $A$?
%
\end{enumerate}
\end{question}

Question \ref{xxque0.2}(2,4) was partially answered when $A$ is a
commutative domain by Etingof-Walton in \cite[Theorem 1.3]{EW}
and when $A$ is a skew polynomial ring with generic parameters
in \cite[Theorems 0.4 and 4.3]{CWZ}.
These questions arise in other subjects such as
theory of infinite dimensional Hopf algebras, study of braided 
Hopf algebras and Nichols
algebras, study of skew (or twisted) Calabi-Yau algebras, noncommutative
algebraic geometry and representation theory of quivers. Another interesting
question is \cite[Question 5.9]{EW}: If $H$ is semisimple and $A$ is PI
(namely, $A$ satisfies a polynomial identity), must then
${\text{PIdeg}}(H^{\circ}) \leq {\text{PIdeg}}(A)^2$?
This paper answers some of these questions in special cases.

One basic idea (or method) in this paper is the homological identity given in
\cite[Theorem 0.1]{CWZ} which roughly says that the Nakayama automorphism
controls, in some aspects, the class of Hopf algebras that act on a
given Artin-Schelter (or AS, for short) regular algebra $A$. 
If we know the Nakayama
automorphism of $A$, then we can grab information about this class of
Hopf algebras. Examples of such results are given in
\cite[Theorems 0.4 and 0.6]{CWZ}.
One immediate consequence of \cite[Theorem 0.6]{CWZ} is the following.

\begin{corollary}
\label{xxcor0.3}\cite[Comment after Theorem 0.6]{CWZ}
Let $A$ be the 3-dimensional Sklyanin algebra and
let $H$ be a finite dimensional Hopf algebra acting on $A$ inner
faithfully with trivial homological determinant.
Then $H$ is semisimple.
\end{corollary}

One interesting question is to classify all possible
Hopf algebras that act inner faithfully on a 3-dimensional (PI) 
Sklyanin algebra.
Corollary \ref{xxcor0.3} holds because the Nakayama automorphism of the
3-dimensional Sklyanin algebra is the identity map, which satisfies
the hypotheses of \cite[Theorem 0.6]{CWZ}.
For this reason, it is important to understand and describe  explicitly
the Nakayama  automorphism.

Several authors calculated the  Nakayama automorphism of skew (or twisted)
Calabi-Yau algebras \cite{BrZ, GY, LWW, RRZ1, RRZ2, Ye}. In general the Nakayama
automorphism is a subtle invariant and is  difficult to compute. Several
researchers have been investigating the Nakayama automorphism and its
applications. In \cite{RRZ1, RRZ2} Rogalski-Reyes-Zhang proved several
homological identities about the Nakayama automorphism. Liu-Wang-Wu
studied the Nakayama automorphism for Ore extensions; in
particular, they gave a description of the Nakayama automorphism
of $A[x;\sigma,\delta]$ \cite[Theorem 0.2]{LWW}.

This paper will focus on the following classes of AS regular algebras:

$A(1)=k_{p_{ij}}[t_1,t_2,t_3]$, where $p_{ij}\in k^\times:=k\setminus \{0\}$,
is generated by $t_1,t_2,t_3$ and subject to the relations
\begin{equation}
\label{E0.3.1}\tag{E0.3.1}
t_j t_i=p_{ij} t_i t_j,\quad {\text{for all }}\; 1\leq i<j\leq 3.
\end{equation}

$A(2)$ is generated by $t_1,t_2,t_3$ and subject to the relations
\begin{equation}
\label{E0.3.2}\tag{E0.3.2}
t_1t_2-t_2t_1=t_1t_3-t_3t_1=t_3t_2-pt_2t_3-t_1^2=0
\end{equation}
where $p\in k^\times $.

$A(3)$ is generated by $t_1,t_2,t_3$ and subject to the relations
\begin{equation}
\label{E0.3.3}\tag{E0.3.3}
(t_2+t_1) t_1-t_1 t_2=t_3 t_1-q t_1 t_3=t_3 t_2-q (t_2+t_1)t_3=0
\end{equation}
where $q\in k^\times $.

$A(4)$ is generated by $t_1,t_2,t_3$ and subject to the relations
\begin{equation}
\label{E0.3.4}\tag{E0.3.4}
(t_2+t_1)t_1-t_1t_2= t_3t_1-p t_1t_3=t_3t_2-p t_2t_3=0
\end{equation}
where $p\in k^\times $.

$A(5)$ is generated by $t_1,t_2,t_3$ and subject to the relations
\begin{equation}
\label{E0.3.5}\tag{E0.3.5}
(t_2+t_1)t_1-t_1t_2=(t_3+t_2+t_1)t_1-t_1t_3=(t_3+t_2+t_1)t_2-(t_2+t_1)t_3=0.
\end{equation}

$A(6)$ is the graded down-up algebra $A(\alpha, \beta)$ which
is generated by $x,y$ and subject to the relations
\begin{equation}
\label{E0.3.6}\tag{E0.3.6}
x^2y-\alpha xyx -\beta yx^2=xy^2-\alpha yxy-\beta y^2x=0
\end{equation}
where $\alpha \in k$ and $\beta\in k^\times$.

$A(7):=S(p)$ is generated by $x,y$ and subject to the relations
\begin{equation}
\label{E0.3.7}\tag{E0.3.7}
x^2y-pyx^2= xy^2+py^2x=0
\end{equation}
where $p\in k^\times$.

These are some pair-wise non-isomorphic noetherian AS regular 
algebras of global dimension three. The Nakayama automorphism of 
these algebras will be given explicitly in Section \ref{xxsec1}.

Going back to noncommutative invariant theory, we adapt the standard
hypotheses given in \cite{CWZ}. The first application of
the Nakayama automorphism is the following result which answers
Questions \ref{xxque0.1} and  \ref{xxque0.2}(2) in some special cases.

\begin{theorem}
\label{xxthm0.4}
Let $A$ be one of the following algebras and $H$ act on $A$
satisfying Hypothesis {\rm{\ref{xxhyp2.1}}}.
Then $H$ is a commutative group algebra.
\begin{enumerate}
\item[(1)]
$A(1)$ where elements $p_{12}^{-2} p_{23}p_{31}$, $p_{31}^{-2} p_{12}p_{23}$
and $p_{23}^{-2} p_{31}p_{12}$ are not roots of unity. A special case is
when $p_{12}$ and $p_{13}$ are roots of unity and $p_{23}$ is not.
\item[(2)]
$A(2)$ where $p$ is not a root of unity.
\item[(3)]
$A(3)$ where $q$ is not a root of unity.
\item[(4)]
$A(4)$ where $p$ is not a root of unity.
\item[(5)]
$A(5)$.
\item[(6)]
$A(6)=A(\alpha, \beta)$ where $\alpha\neq 0$ and $\beta$ is not a root of unity.
\end{enumerate}
\end{theorem}

The commutative group algebra $H$ in Theorem \ref{xxthm0.4} can be
described explicitly. Theorem \ref{xxthm0.4} also holds for the following AS
regular algebras.

\begin{enumerate}
\item[(0.4.8)]
The class ${\bf D}$ of AS regular algebras of global
dimension five defined in \cite{LWW} with generic parameters.
\item[(0.4.9)]
The class ${\bf G}$ of AS regular algebras of global
dimension five defined in \cite{LWW} with generic parameters.
%
\end{enumerate}

Note that if $A=A(0,\beta)$ (or $S(p)$),
then there are non-group actions on $A$ [Proposition \ref{xxpro2.16}].
AS regular algebras of global dimension three were classified by Artin,
Schelter, Tate and Van den Bergh \cite{AS, ATV1, ATV2}. The algebras
listed before Theorem \ref{xxthm0.4} are only a subset in their classification,
nevertheless, Theorem \ref{xxthm0.4} answers Question \ref{xxque0.1}
in some special cases.  Note that all algebras in Theorem \ref{xxthm0.4}
are not PI.  In the PI case, we have the following result when $H$ is
semisimple and the $H$-action
has trivial homological determinant, which partially answers
Question \ref{xxque0.2}(4).

\begin{theorem}
\label{xxthm0.5}
Assume that $A$ is one of the following algebras. Let $H$ be a
semisimple Hopf algebra acting on $A$ and satisfying Hypothesis
{\rm{\ref{xxhyp2.1}}}. Suppose the $H$-action has trivial
homological determinant. Then $H$ is the dual of a group algebra.
\begin{enumerate}
\item[(1)]
$A(1)$ where $p_{12}^{-2} p_{23}p_{31}$, $p_{31}^{-2} p_{12}p_{23}$ and
$p_{23}^{-2} p_{31}p_{12}$ are not equal to $1$.
\item[(2)]
$A(2)$ where $p\neq \pm 1$.
\item[(3)]
$A(3)$ {\rm{(}}even without ``trivial homological determinant'' 
hypothesis{\rm{)}}.
\item[(4)]
$A(4)$ {\rm{(}}even without ``trivial homological determinant'' hypothesis{\rm{)}}.
\item[(6)]
$A(6)$ where $\beta\neq \pm 1$.
\item[(7)]
$A(7)$ where $p\neq \pm i$.
\end{enumerate}
In cases {\rm{(1)-(4)}}, $H$ is also a commutative group algebra.
\end{theorem}

Theorem \ref{xxthm0.5} also holds for the following algebras.
\begin{enumerate}
\item[(0.5.8)]
The class ${\bf D}$ of AS regular algebras of global
dimension five defined in \cite{LWW} with $p^{-6}q^8\neq 1$.
\item[(0.5.9)]
The class ${\bf G}$ of AS regular algebras of global
dimension five defined in \cite{LWW} with $g\neq \pm 1$.
\end{enumerate}

See Corollaries \ref{xxcor2.6} and \ref{xxcor2.7} for other
related results.

The second application of the Nakayama automorphism relies on the
following theorem.

\begin{theorem}
\label{xxthm0.6} Let $A$ be a connected graded domain.
Then its  Nakayama automorphism {\rm{(}}if exists{\rm{)}}
commutes with any algebra automorphism of $A$.
\end{theorem}

A slightly weaker version of Theorem \ref{xxthm0.6} was proved
in \cite[Theorem 3.11]{RRZ1}, which states that $\mu_A$
commutes with graded algebra automorphisms of $A$.
As a consequence of Theorem \ref{xxthm0.6} we have the following,
which partially answers the second part of Question \ref{xxque0.2}(5).

\begin{corollary}
\label{xxcor0.7} If $A$ is one of the following algebras, then
every algebra automorphism of $A$ preserves the ${\mathbb N}$-grading
of $A$.
\begin{enumerate}
\item[(1)]
$A(1)$ where $p_{ij}$ are generic or $(p_{12},p_{13},p_{23})=(1,1,p)$ with
$p$ not a root of unity.
\item[(2-4)]
$A(2)$-$A(4)$ with the same hypotheses as in Theorem {\rm{\ref{xxthm0.4}(2-4)}}.
\item[(6)]
$A(6)$ where $\beta$ is not a root of unity.
\item[(7)]
$A(7)$ where $p$ is not a root of unity.
\end{enumerate}
\end{corollary}

Using Corollary \ref{xxcor0.7}, the automorphism group
of the algebras are calculated explicitly, see Section \ref{xxsec4}.
Note that there are non-graded algebra automorphisms for algebra $A(5)$,
see Theorem \ref{xxthm5.8}. 
A much more interesting and difficult question
is the first part of Question \ref{xxque0.2}(5). For example, we ask
whether every finite dimensional Hopf action on the algebras in
Corollary \ref{xxcor0.7} preserve the ${\mathbb N}$-grading.

In the papers \cite{CPWZ1, CPWZ2, BeZ}
the authors use the discriminant to control algebra automorphisms and locally
nilpotent derivations. The third application of the Nakayama automorphism
concerns with locally nilpotent derivations and the cancellation problem.

\begin{theorem}
\label{xxthm0.8} Let $A$ be a connected graded domain.
Then its  Nakayama automorphism {\rm{(}}if exists{\rm{)}}
commutes with  any locally nilpotent derivation of $A$.
\end{theorem}

Using the similar ideas to those presented in \cite{BeZ}, we can solve
Zariski Cancellation Problem for some noncommutative
algebras. Recall that an algebra $A$ is called {\it cancellative}
if $A[t]\cong B[t]$ for any algebra $B$ implies that $A\cong B$.
The original Zariski Cancellation Problem asks if the commutative
polynomial ring $k[t_1,\cdots,t_n]$ is cancellative. Recall that
$k[t_1]$ is cancellative by a result of Abhyankar-Eakin-Heinzer
\cite{AEH}, $k[t_1,t_2]$ is cancellative by Fujita \cite{Fu} and 
Miyanishi-Sugie \cite{MS} in characteristic zero and by Russell
\cite{Ru} in positive characteristic. The Zariski Cancellation Problem
was open for many years. In 2013, a remarkable development was made by
Gupta \cite{Gu1,Gu2} who completely settled this problem negatively
in positive characteristic for $n\geq 3$. The Zariski Cancellation Problem
in characteristic zero remains open for $n\geq 3$. Note that the
Zariski Cancellation Problem is also related to Jacobian Conjecture.
See \cite{BeZ} for some background about cancellation problems.
Here we solve the Zariski Cancellation Problem (ZCP) for some
noncommutative algebras of dimension three.

\begin{corollary}
\label{xxcor0.9} Let $A$ be any algebra  in Corollary
{\rm{\ref{xxcor0.7}}}. Then $A$ is  cancellative.
\end{corollary}

The paper is organized as follows. In Section \ref{xxsec1},
we recall the definition of
the Nakayama automorphism and AS regular algebras and compute the Nakayama
automorphism of the algebras in Theorem \ref{xxthm0.4}. The proofs of
Theorem \ref{xxthm0.4} and Theorem \ref{xxthm0.5} are given in Sections \ref{xxsec2} and
\ref{xxsec3}. In Section \ref{xxsec4}, we prove Theorem \ref{xxthm0.6}
and Corollary \ref{xxcor0.7}. In Section 5, we calculate the  
full automorphism group of $A(5)$. In Section \ref{xxsec5}, we prove 
Theorem \ref{xxthm0.8} and Corollary \ref{xxcor0.9}.

\section{Nakayama automorphism}
\label{xxsec1}
Throughout let $A$ be an algebra over $k$. Let $A^e$ denote the
enveloping algebra $A\otimes A^{op}$, where $A^{op}$ is the opposite ring
of $A$. An $A$-bimodule can be identified with a left $A^e$-module.

\begin{definition}\cite[Definition 0.1]{RRZ1}
\label{xxdef1.1} Let $A$ be an algebra over $k$.
\begin{enumerate}
\item[(a)]
$A$ is called {\it skew Calabi-Yau} (or {\it skew CY}, for short) if
\begin{enumerate}
\item[(i)]
$A$ is homologically smooth, that is, $A$ has a projective resolution
in the category $A^e$-$\Mod$ that has finite length and such that each
term in the projective resolution is finitely generated, and
\item[(ii)]
there is an integer $d$ and an algebra automorphism $\mu$ of $A$ such that
\begin{equation}
\label{E1.1.1}\tag{E1.1.1}
\Ext^i_{A^e}(A,A^e)\cong \begin{cases} 0 & i\neq d\\
{^1 A^\mu} & i=d,
\end{cases}
\end{equation}
as $A$-bimodules, where $1$ denotes the identity map of $A$.
\end{enumerate}
\item[(b)]
If \eqref{E1.1.1} holds for some algebra automorphism $\mu$ of
$A$ (even if $A$ is not skew CY), then $\mu$ is called the
{\it Nakayama automorphism} of $A$, and is usually denoted by $\mu_A$.
\item[(c)]
\cite[Definition 3.2.3]{Gi}
We call $A$ {\it Calabi-Yau} (or {\it CY}, for short) if $A$ is skew
Calabi-Yau and $\mu_A$ is inner.
\end{enumerate}
\end{definition}

\begin{definition}
\label{xxdef1.2}
A connected graded algebra $A$ is called
{\it Artin-Schelter Gorenstein} (or {\it AS Gorenstein},
for short) if the following conditions hold:
\begin{enumerate}
\item[(a)]
$A$ has finite injective dimension $d<\infty$ on both sides,
\item[(b)]
$\Ext^i_A(k,A) = \Ext^i_{A^{op}} (k,A) = 0$ for all
$i \neq d$ where $k = A/A_{\geq 1}$, and
\item[(c)]
$\Ext^d_A(k,A) \cong k(l)$ and $\Ext^d_{A^{op}} (k,A)
\cong k(l)$ for some integer $l$. The integer
$l$ is called the AS index.
\end{enumerate}
If moreover
\begin{enumerate}
\item[(d)] $A$ has (graded) finite global dimension $d$,
then $A$ is called {\it Artin-Schelter regular} (or
{\it AS regular}, for short).
\end{enumerate}
\end{definition}

By \cite[Lemma 1.2]{RRZ1}, if $A$ is connected graded, then $A$ is AS
regular if and only if $A$ is skew CY.

We will use the following two lemmas several times.

\begin{lemma}
\label{xxlem1.3}\cite[Lemma 1.5]{RRZ1}
Let $A$ be a noetherian connected graded AS Gorenstein algebra and
let $z$ be a homogeneous normal regular element of positive degree
such that $\mu_A(z)=cz$ for some $c\in k^\times$. Let $\tau$  be
in $\Aut(A)$ such that $za = \tau(a)z$ for all $a\in A$.
Then $\mu_{A/(z)}$ is equal to $\mu_A \circ \tau$ when
restricted to $A/(z)$.
\end{lemma}

Given a nonzero scalar $c\in k^\times$, we define a graded algebra 
automorphism $\xi_c$ of $A$  by
$$\xi_c(a)=c^{\deg a} a$$
for all homogeneous elements $a\in A$.

\begin{lemma}
\label{xxlem1.4}\cite[Theorem 0.3]{RRZ1}
Let $A$ be a noetherian connected graded AS Gorenstein algebra
with AS index $l$. Let $A^\sigma$ be the graded twist of $A$
associated to a graded automorphism $\sigma$ of $A$.
Then $\mu_{A^{\sigma}}=\mu_A \circ \sigma^l \circ \xi_{\hdet (\sigma)}^{-1}$.
\end{lemma}

Lemma \ref{xxlem1.4} holds in the multi-graded case, see
\cite[Theorem 5.4(a)]{RRZ1}. We now compute the Nakayama automorphism
of some classes of AS regular algebras of dimension three (or higher).

First we consider the skew polynomial ring
$A=k_{p_{ij}}[t_1,\cdots.t_n]$ which is generated by
$t_1,\cdots,t_n$ and subject to the relations
$$t_j t_i=p_{ij} t_i t_j$$
for all $i<j$, where $p_{ij}\in k^\times$ are nonzero scalars satisfying
$p_{ii}=1$ for all $i$ and $p_{ji}=p_{ij}^{-1}$ for all $i,j$.
The Nakayama automorphism of $A$ is calculated in \cite[Proposition 4.1]{LWW}
and \cite[Example 5.5]{RRZ1}:
$$\mu_A: t_i\mapsto 
(\prod_{s=1}^{n} p_{si}) t_i, \quad {\text{for all }}\; i.$$
If $n$ is three, then the Nakayama
automorphism of $k_{p_{ij}}[t_1,t_2,t_3]$ is determined by
\begin{equation}
\label{E1.5.1}\tag{E1.5.1}
\mu_A: t_1\to p_{21}p_{31} \;t_1,\quad
t_2\to p_{12}p_{32}\;t_2,\quad
t_3\to p_{13}p_{23} \;t_3.
\end{equation}

The next algebra is
$$A=k\langle t_1,t_2,\cdots, t_n\rangle/
(t_nt_{n-1}-pt_{n-1}t_n-\sum_{i=1}^s t_i^2,
t_i \; {\text{is central for all}}\; i\leq n-2)$$
where $p\in k^\times $ and $s$ is a positive integer no more than
$n-2$. Let $\Omega=
\sum_{i=1}^s t_i^2$. Then $\Omega$ is central and $\mu_A(\Omega)=\Omega$.
By Lemma \ref{xxlem1.3},
$$\mu_A=\mu_{A/(\Omega)}=\mu_{B/(\Omega)}=\mu_{B}$$
where
$$B=k\langle t_1,t_2,\cdots, t_n\rangle/
(t_nt_{n-1}-pt_{n-1}t_n,
t_i \; {\text{is central for all}}\; i\leq n-2).$$
Hence
\begin{equation}
\label{E1.5.2}\tag{E1.5.2}
\mu_A: t_i\to t_i, \; \forall \; i\leq n-2, \quad{\text{and}}\quad
t_{n-1}\to p^{-1} t_{n-1}, \; t_n\to p t_n
\end{equation}
If $n=3$ and $s=1$, then
$$\mu_A: t_1\to t_1, \quad t_2\to p^{-1} t_2, \quad t_3\to p t_3.$$

Consider the commutative polynomial
ring $B=k[t_1,t_2,t_3]$ as a ${\mathbb Z}$-graded algebra with
$\deg t_i=1$ for all $i=1,2,3$. Let $\sigma$ be a graded algebra
automorphism of $B$ determined by $t_1\to t_1, t_2\to t_2-t_1, t_3\to qt_3$
where $q$ is a nonzero scalar. We use the convention in \cite[Section 5]{RRZ1}
to deal with the graded twist, so we can use the identity proved
in \cite[Theorem 5.4(a)]{RRZ1}. The new multiplication of the graded
twist, denoted by $B^{\sigma}$, associated to $\sigma$ is defined by
\cite[(E5.0.2)]{RRZ1}. Starting from a commutative relation $t_j t_i=t_i
t_j$, we have a relation $\sigma^{-1}(t_j) t_i=\sigma^{-1}(t_i) t_j$ in
the graded twist $B^{\sigma}$. Hence the relations for the algebra $B^{\sigma}$
are
$$(t_2+t_1) t_1=t_1 t_2, \; q^{-1}t_3 t_1=t_1 t_3, 
\; q^{-1} t_3 t_2=(t_2+t_1) t_3$$
which show that $B^{\sigma}$ is isomorphic to the $A(3)$ given before Theorem
\ref{xxthm0.4}. Note that $l=3$ for the algebra $B$ and 
$\hdet \sigma=q$, using the identity
$$\mu_{B^{\sigma}}=\mu_{B}\circ \sigma^{l}\circ \xi_{\hdet \sigma}^{-1}$$
[Lemma \ref{xxlem1.4}] or \cite[Theorem 5.4(a)]{RRZ1}, we have a 
formula for the Nakayama automorphism of $A(3)$:
\begin{equation}
\label{E1.5.3}\tag{E1.5.3}
\mu_{A(3)}: t_1\to q^{-1} t_1, \; t_2\to q^{-1}(t_2-3t_1), \; t_3\to q^{2} t_3.
\end{equation}

For the next algebra, we need to work with a ${\mathbb Z}^2$-graded twist.
Let $B$ be the commutative polynomial ring $k[t_1,t_2,t_3]$ with 
${\mathbb Z}^2$-grading determined by 
$\deg t_1=\deg t_2=(1,0)$ and $\deg t_3=(0,1)$. Consider two
graded algebra automorphisms $\sigma_1: t_1\to t_1, t_2\to t_2-t_1,
t_3\to t_3$ and $\sigma_2: t_1\to p^{-1} t_1, t_2\to p^{-1} t_2, t_3\to t_3$
where $p$ is a nonzero scalar.
Define a twisting system $\sigma=\{\sigma_{a,b}=\sigma_1^a
\sigma_2^b\mid (a,b)\in {\mathbb Z}^2\}$. The graded twist $B^{\sigma}$
is isomorphic to
$$A(4)=k\langle t_1,t_2,t_3\rangle/((t_2+t_1)t_1-t_1t_2, t_3t_1-p t_1t_3,
t_3t_2-p t_2t_3).$$
Note that $l=(2,1)$ and $\hdet \sigma=(1, p^{-2})$.
By \cite[Theorem 5.4(a)]{RRZ1}, we have a formula for the Nakayama
automorphism
\begin{equation}
\label{E1.5.4}\tag{E1.5.4}
\mu_{A(4)}: t_1\to p^{-1} t_1,\; t_2\to p^{-1}(t_2-2t_1), \; t_3\to p^{2} t_3.
\end{equation}

The algebra
$$A(5):=k\langle t_1,t_2,t_3\rangle/((t_2+t_1)t_1-t_1t_2, (t_3+t_2+t_1)t_1-t_1t_3,
(t_3+t_2+t_1)t_2-(t_2+t_1)t_3)$$
is a graded twist $B^\sigma$ with $\sigma: t_1\to t_1, t_2\to t_2-t_1, t_3\to t_3-t_2$.
So the Nakayama automorphism of $A(5)$ is determined by
\begin{equation}
\label{E1.5.5}\tag{E1.5.5}
\mu_{A(5)}: t_1\to t_1, \; t_2\to t_2-3 t_1, \; t_3\to t_3- 3t_2+ 3t_1.
\end{equation}

Graded down-up algebras $A(\alpha,\beta)$ have been studied by several
researchers. By definition,
$A(\alpha, \beta)$ is generated by $x$ and $y$ and subject to the relations
$$x^2 y-\alpha xyx-\beta yx^2, \quad
xy^2-\alpha yxy-\beta y^2 x.$$
This is an AS regular algebra when $\beta\neq 0$. This class of algebras
are not Koszul, but $3$-Koszul.
Graded automorphisms of $A(\alpha,\beta)$ have been worked out in
\cite{KK}. Consider the characteristic equation
$$w^2-\alpha w-\beta =0$$
and let $w_1$ and $w_2$ be the roots of the above equation.
Then $\Omega:=xy-w_1yx$ is a normal regular element and
$A(\alpha,\beta)/(\Omega)$ is a skew polynomial ring of
global dimension two. Since we know the Nakayama automorphism
of $A(\alpha,\beta)/(\Omega)$ by \cite[Proposition 4.1]{LWW},
Lemma \ref{xxlem1.3} shows that the Nakayama automorphism
of $A(6):=A(\alpha,\beta)$ is determined by
\begin{equation}
\label{E1.5.6}\tag{E1.5.6}
\mu_{A(6)}: x\to -\beta x, \quad y\to -\beta^{-1}y.
\end{equation}

Another class of non-Koszul AS regular is
$A(7):=S(p)$ which is generated by $x$ and $y$ and subject to
relations
$$x^2y-pyx^2, \quad xy^2+py^2x$$
where $p\in k^\times $. Note that $z:=y^2$ is a normal element
such that $\mu_{A(7)}(z)=c z$ since the Nakayama automorphism
preserves ${\mathbb Z}^2$-grading of $A(7)$. It is easy to see
that $ za=\tau(a) z$ where $\tau\in \Aut(A(7))$ maps $x$ to
$-p^{-1} x$ and $y$ to $y$. By Lemma \ref{xxlem1.3},
$\mu_{A(7)/(z)}=\mu_{A(7)}\circ \tau$. Applying Lemma \ref{xxlem1.3}
to the algebra $A(6)$ with $z=y^2$, we have
$\mu_{A(6)/(z)}=\mu_{A(6)}\circ \tau'$ where $\tau'$ maps
$x$ to $\beta^{-1}x$ and $y$ to $y$. Note that $A(6)/(y^2)\mid_{\alpha=0,\beta=p}
=A(7)/(y^2)$. Therefore, when $\alpha=0$ and $\beta=p$,
$$\mu_{A(7)}\circ \tau=\mu_{A(7)/(z)}=\mu_{A(6)/(z)}=\mu_{A(6)}\circ \tau'.$$
Then, by an easy calculation and \eqref{E1.5.6}, we obtain
that
\begin{equation}
\label{E1.5.7}\tag{E1.5.7}
\mu_{A(7)}: x\to p x, \quad y\to -p^{-1}y.
\end{equation}

Finally, let us mention two classes of AS regular algebras of dimension five
for which the Nakayama automorphism has been computed by Liu-Wang-Wu 
\cite{LWW}. As in \cite{LWW}, the algebras ${\bf D}$
and ${\bf G}$ are of the form $k\langle x,y\rangle (r_1,r_2,r_3)$
where $r_i$ are relations. For the algebra ${\bf D}$, the three
relations are
$$\begin{aligned}
r_{D1}&=x^3 y+ p x^2 y x+ q x y x^2 -p(2p^2+q) yx^3,\\
r_{D2}&=x^2 y^2 -p(p^2+q) yxyx - q^2 y^2 x^2
+(q-p^2) x y^2 x+(q-p^2) y x^2 y,\\
r_{D3}&= x y^3 + py x y^2 + q y^2 x y-p(2p^2+q) y^3 x,
\end{aligned}
$$
where $p,q\in k^\times$ and $2 p^4-p^2q+q^2=0$. By
\cite[Theorem 4.3(1)]{LWW},
the Nakayama automorphism of ${\bf D}$ is given by
\begin{equation}
\label{E1.5.8}\tag{E1.5.8}
\mu_{\bf D}: x\to p^{-3} q^{4} x,\quad y\to p^{3}q^{-4} y.
\end{equation}

For the algebra ${\bf G}$, the  three relations are
$$\begin{aligned}
r_{G1}&=x^3 y+ px^2 y x + q x y x^2 +s y x^3,\\
r_{G2}&= x^2 y^2+l_2 xyxy +l_3 yxyx +l_4 y^2 x^2 +l_5 xy^2 x +l_5 yx^2 y,\\
r_{G3}&=xy^3+p y x y^2 + q y^2 x y+s y^3 x,
\end{aligned}
$$
where
$$l_2=\frac{-s^2(qs-g)}{g(qs+g)}, \quad
l_3=s-\frac{pg(ps-q^2)}{q(qs+g)}, \quad
l_4=\frac{-g^2}{s^2}, \quad l_5=\frac{ps^2+qg}{qs+g},
$$
with $p,q,s,g\in k^\times$, $ps^3g+qsg^2+s^5+g^3=0$, $p^3s=q^3$,
$ps\neq q^2$, $q^2s^2\neq g^2$ and $s^5+g^3\neq 0$.
By \cite[Theorem 4.3(2)]{LWW},
the Nakayama automorphism of ${\bf G}$ is given by
\begin{equation}
\label{E1.5.9}\tag{E1.5.9}
\mu_{\bf G}: x\to g  x, \quad y\to g^{-1} y.
\end{equation}

The above are all AS regular algebras that we will be dealing with.

\section{Hopf actions on $A$ with diagonalizable $\mu_A$}
\label{xxsec2}

In this and the next sections we study finite dimensional Hopf actions on
AS regular algebras and partially answer Questions \ref{xxque0.1} and
\ref{xxque0.2}(2,4). We impose the following standard hypotheses for
the rest of the section.

\begin{hypothesis}\cite[Hypothesis 0.3]{CWZ}
\label{xxhyp2.1} We assume that
\begin{enumerate}
\item[(i)]
$H$ is a finite dimensional Hopf algebra.
\item[(ii)]
$A$ is a connected graded AS regular algebra generated in degree 1.
\item[(iii)]
$H$ acts on $A$ inner faithfully, namely,  there is no nonzero Hopf ideal
$I\subset H$ such that $IA= 0$.
\item[(iv)]
The $H$-action on $A$ preserves the grading of A.
\end{enumerate}
\end{hypothesis}

Let $A$ be an AS regular algebra generated in degree 1
and let $V:=A_1$ be the degree 1 piece of $A$. Let $K:=H^{\circ}$
be the dual Hopf algebra of $H$. Then a left $H$-action on $A$ is
equivalent to a right $K$-coaction on $A$.

We say that a right $K$-coaction on $A$ is inner faithful if for
any proper Hopf subalgebra $K'\subsetneq K$,  we have that
$\rho(A)\nsubseteq A\otimes K'$. The main tool of this paper is the
following homological identity proved in \cite[Theorem 0.1]{CWZ}
with an improved version given in \cite{RRZ2}.

\begin{theorem}\cite[Theorem 4.3]{RRZ2}
\label{xxthm2.2}
Let $A$ be a noetherian
AS regular algebra with Nakayama automorphism $\mu_A$.
Let $K$ be a Hopf algebra with bijective antipode $S$ coacting on $A$
inner faithfully from the right. Suppose that the homological codeterminant
\cite[Definition 1.5(b)]{CWZ} of the $K$-coaction on $A$ is the element
$D \in K$. Then
\begin{equation}
\label{E2.2.1}\tag{E2.2.1}
\eta_{D}\circ S^2=\eta_{\mu_{A}^{\tau}}
\end{equation}
where $\eta_D$ is the automorphism of $K$ defined by conjugating by $D$ and
$\eta_{\mu_A^{\tau}}$ is the automorphism of $K$ given by conjugating by
the transpose of the corresponding matrix of $\mu_A$.
\end{theorem}

Here are some explanations from \cite{CWZ}.
The automorphism on the left-hand side of equation \eqref{E2.2.1} is the
composition of the Hopf algebra automorphism $S^2$ of $K$ (which is bijective
by hypothesis) and the Hopf algebra automorphism $\eta_D$ of $K$ where
$\eta_D$ is given by $\eta_D(a) = D^{-1} aD$ for all $a\in K$. To understand the
right-hand side of equation \eqref{E2.2.1} we start with a $k$-linear basis,
say $\{v_1, \cdots, v_n\}$ of $A_1$, the degree $1$ graded piece of
$A$. Then the Nakayama automorphism $\mu_A$ can be written as
\begin{equation}
\label{E2.2.2}\tag{E2.2.2}
\mu_A(v_i)=\sum_{i=1}^{n} m_{ij} v_j, \quad {\text{for all }}\; i=1,\cdots, n.
\end{equation}
Let ${\mathbb M}$ be the $n \times n$-matrix $(m_{ij})_{n\times n}$ over the base
field $k$.  Let $\rho: A\to A\otimes K$ be the right $K$-coaction on $A$. Then
\begin{equation}
\label{E2.2.3}\tag{E2.2.3}
\rho(v_i)=\sum_{j} v_j\otimes y_{ji},\quad {\text{for all }}\; i=1,\cdots, n
\end{equation}
for some $y_{ji}\in K$. Then $\Delta(y_{st})=\sum_{j=1}^n y_{sj}
\otimes y_{jt}$ and $\epsilon(y_{st})=\delta_{st}$ for all $s,t$.
Let $\rho^*$ be the left $K$-coaction on  the $\Ext$-algebra $E :=\Ext^*_A(k,k)$
induced by the $K$-coaction on $A$. Then we  have
$$\rho^*(v_i^*)=\sum_{s=1}^n y_{is}\otimes v_s^*,
\quad {\text{for all }}\; i=1,\cdots, n.$$
Since the $K$-coaction on $A$ is inner faithful,
$\{y_{ij}\}_{1\leq i,j \leq n}$ generates $K$ as a Hopf algebra.
With this choice of basis $\{v_i\}_{i=1}^n$, we define
\begin{equation}
\label{E2.2.4}\tag{E2.2.4}
\eta_{\mu_A^\tau}: y_{ij}\to \sum_{s,t=1}^n m_{si} y_{st} w_{jt}
\end{equation}
for all $1 \leq i, j \leq n$ where the matrix
${\mathbb W}:=(w_{ij})_{n\times n}=(m_{ij})_{n\times n}^{-1}$.
Roughly speaking, we use coordinates to define the conjugation automorphism
$\eta_{\mu_A^{\tau}}$ of $K$. Theorem \ref{xxthm2.2} implies that 
the definition of
this automorphism is independent of the choice of coordinates.

Since $K$ is finite dimensional, $\eta_D\circ S^2$ has finite order. So
it follows from \eqref{E2.2.1} that $\eta_{\mu_A^{\tau}}$ is of finite order.
In fact, the order of $\eta_{\mu_A^{\tau}}$ divides $ 2\dim_k K$,
see the proof of \cite[Theorem 4.3]{CWZ}.

In this section, we consider the case when ${\mathbb M}$ is a diagonal matrix.
Let $A$ be a noetherian AS regular algebra generated by $A_1=V$.
We say the Nakayama automorphism $\mu_A$ has eigenvalues
$\{\lambda_1,\cdots,\lambda_{n}\}$ if there is a $k$-linear basis
$\{v_1,\cdots,v_n\}$ of $V$ such that $\mu(v_i)=\lambda_i v_i$
for all $i=1,\cdots,n$. From now on we assume Hypothesis \ref{xxhyp2.1}.

\begin{lemma}
\label{xxlem2.3} Let $A$ be an AS regular algebra
generated by $V= A_1$ and $\mu_A$ be the Nakayama
automorphism of $A$. Suppose $\mu_A$ has eigenvalues
$\{\lambda_1,\cdots,\lambda_{n}\}$.
\begin{enumerate}
\item[(1)]
If $\lambda_i \lambda_j^{-1}$ is not a root of unity for a
pair $(i,j)$, then $y_{ij}=y_{ji}=0$.
\item[(2)]
If $\lambda_i \lambda_j^{-1}$
is not a root of unity for all $i\neq j$, then $K$ is a group algebra.
As a consequence, $H$ and $K$ are semisimple.
\item[(3)]
Suppose the setting of part {\rm{(2)}}.
If, further, $v_jv_i=p_{ij} v_i v_j\neq 0$ for all $i\neq j$
for some $p_{ij}\in k^\times$, then $H$ and $K$ are commutative group
algebras.
\end{enumerate}
For the next three parts, we further assume that $H$ is semisimple
and the $H$-action on $A$ has  trivial homological determinant.
\begin{enumerate}
\item[(4)]
If $\lambda_i \lambda_j^{-1}$ is not 1 for a
pair $(i,j)$, then $y_{ij}=y_{ji}=0$.
\item[(5)]
If $\lambda_i \lambda_j^{-1}$
is not 1 for all $i\neq j$, then $K$ is a group algebra.
As a consequence, $H$ and $K$ are semisimple.
\item[(6)]
Suppose the setting of part {\rm{(5)}}.
If, further, $v_jv_i=p_{ij} v_i v_j\neq 0$ for all $i\neq j$
for some $p_{ij}\in k^\times$, then $H$ and $K$ are  commutative group
algebras.
\end{enumerate}
\end{lemma}
\begin{proof} (1)  By \eqref{E2.2.4},
$$\eta_{\mu_A^{\tau}}(y_{ij})=\lambda_i\lambda_j^{-1} y_{ij}.$$
By Theorem \ref{xxthm2.2}, $\eta_{D}\circ S^2(y_{ij})=
\eta_{\mu_A^{\tau}}(y_{ij})$ for all $i,j$. Since $K$ is finite
dimensional, the automorphism $\eta_{D}\circ S^2$ has
finite order. Hence, $\eta_{\mu_A^{\tau}}$ has finite
order, and therefore there is an $N$ such that
$$y_{ij}=(\eta_{\mu_A^{\tau}})^N(y_{ij})=(\lambda_i\lambda_j^{-1})^N
y_{ij}.$$
Since $\lambda_i\lambda_j^{-1}$ is not a root of unity,
$y_{ij}=0$. By symmetry, $y_{ji}=0$.

(2) If $\lambda_i\lambda_j^{-1}$ is not a root of unity for all $i\neq j$,
then $y_{ij}=0$ for all $i\neq j$ by part (1). Thus $\rho(v_i)=
v_i\otimes y_{ii}$ and $y_{ii}$ is a group-like element. Since the $K$-coaction
on $A$ is inner faithful, $K$ is generated by $\{y_{ii}\}_{i=1}^n$.
So $K$ is a group algebra.

(3) Applying $\rho$ to the equation $v_j v_i=p_{ij} v_i v_j$, we obtain
that
$$v_j v_i\otimes y_{jj} y_{ii}=\rho(v_jv_i)=p_{ij}
\rho(v_i v_j)=p_{ij}v_i v_j\otimes y_{ii}y_{jj}=
v_jv_i\otimes y_{ii}y_{jj}.$$
Hence $y_{ii}$ commutes with $y_{jj}$.
So $K$ is a commutative group algebra. By \cite[Theorem 2.3.1]{Mo},
$H$ is a commutative group algebra.

(4-6) We have $S^2=Id$ and $D=1_{K}$ since $H$ (and hence $K$) is 
semisimple and the $H$-action
has  trivial homological determinant. By Theorem
\ref{xxthm2.2}, $\eta_{\mu_{A}^{\tau}}$ is the identity map of $K$.
Recall that $\eta_{\mu_A^{\tau}}(y_{ij})=\lambda_i\lambda_j^{-1} y_{ij}$.
If $\lambda_i \lambda_j\neq 1$, then $y_{ij}=0$. The rest of the argument
is similar to the proof of (1-3).
\end{proof}

\begin{proposition}
\label{xxpro2.4}
Let $A$ and $B$ be AS regular algebras generated in degree 1
such that $A\otimes B$ is noetherian. Assume that
the Nakayama automorphisms $\mu_A$ and $\mu_B$
have eigenvalues $\{\lambda_{1},\cdots,\lambda_{m}\}$
and $\{\lambda_{m+1},\cdots,\lambda_{m+n}\}$ respectively.
Suppose that
\begin{enumerate}
\item[(1)]
The only Hopf {\rm{(}}resp. semisimple Hopf{\rm{)}} actions
on $A$ and $B$ are group actions.
\item[(2)]
$\lambda_{i}\lambda_{j}^{-1}$ are not roots of unity
for all $i\leq m$ and $j> m$.
\end{enumerate}
Then every Hopf {\rm{(}}resp. semisimple Hopf{\rm{)}} action on
$A\otimes B$ is a group action.
\end{proposition}
\begin{proof} Let $H$ act on $A\otimes B$ satisfying
Hypothesis \ref{xxhyp2.1}. Then $K:=H^{\circ}$ coacts on $A$.
Since $\mu_{A\otimes B}=\mu_A\otimes \mu_B$, there
is a basis $\{v_1,\cdots,v_m\}\cup\{v_{m+1},\cdots,v_{m+n}\}$
such that $\mu_{A\otimes B}$ has eigenvalues
$\{\lambda_{1},\cdots,\lambda_{m+n}\}$ with respect to
this basis, where $A_1=\sum_{i=1}^m kv_i$ and $B_1=
\sum_{i=m+1}^{m+n} kv_i$. By Lemma \ref{xxlem2.3}(1),  $y_{ij}=0$ if
$(i\leq m, j>m)$ or $(i>m, j\leq m)$.

We make the following remark for the semisimple case. If $H$ is
semisimple, by Larson-Radford \cite{LR1,LR2}, 
$K$ is semisimple, cosemisimple and
involutory (namely, $S^2=Id$). So every Hopf subalgebra of $K$ is
involutory, whence semisimple and cosemisimple.

Let $K_A$ be the Hopf subalgebra of $K$ generated by $\{y_{ij}
\mid 1\leq i,j  \leq m\}$ and $K_B$ be the Hopf subalgebra of $K$
generated by $\{y_{ij}\mid m+1\leq i,j  \leq m+n\}$.
So $K_A$ coacts on $A$ inner faithfully. Now the dual algebra
$(K_A)^{\circ}$ acts on $A$ satisfying Hypothesis
\ref{xxhyp2.1}.  (when $H$ is semisimple, then so is $(K_A)^\circ$.)
By hypothesis (1), $(K_A)^{\circ}$
is a group algebra, so  $K_A$ is commutative.
Similarly, $K_B$ is commutative.
Any $i\leq m$ and $s\geq m+1$, $v_i$ commutes with $v_s$. So, after
applying $\rho$, we have that $y_{ij}$ commutes with $y_{st}$
for all $1\leq i,j \leq m$ and $m+1\leq s,t\leq m+n$.

Finally we will show that $K$ is commutative. Since $K$ is generated
by $K_A$ and $K_B$, it suffices to show that $K_A$ commutes
with $K_B$. Let
$R=M_{n\times n} (K)$ and consider $K$ as the diagonal
subalgebra of $R$.  For any subalgebra $B\subset R$, let
$C_B(K)$ be the centralizer $\{X\in R\mid Xa=aX,
\forall\; a\in B\}$. Let $B$ be the subalgebra
generated by $\{y_{st}\mid m+1\leq s,t \leq m+n\}$.
Let $X=(y_{ij})_{n\times n}$. Then $X\in C_B(K)$.
By the antipode axiom, $X^{-1}=(S(y_{ij}))$.
Since $X^{-1}\in C_{B}(K)$, $S(y_{ij})$
commutes with $y_{st}$ for all $m+1\leq s,t\leq m+n$.
This $K_A$ commutes with $B$. Similarly, one shows that
$K_A$ commutes with $K_B$. Now we have
that $K$ is commutative. Since the base field $k$
is algebraically closed  of characteristic zero, $H=K^\circ$
is a group algebra.
\end{proof}

\begin{remark}
\label{xxrem2.5} Proposition \ref{xxpro2.4} applies to 
the ``twisted''
tensor product $A\otimes_q B$ under suitable hypotheses.
\end{remark}

One immediate consequence is

\begin{corollary}
\label{xxcor2.6}
Suppose $\{p_{ij}\mid 1\leq i<j\leq m\}$ are generic parameters.
Then every semisimple Hopf action on $k_{p_{ij}}[x_1,\cdots,x_m]
\otimes k[y_1,\cdots,y_n]$ is a group action.
\end{corollary}
\begin{proof}  Let $A=k_{p_{ij}}[x_1,\cdots,x_m]$ and
$B=k[y_1,\cdots,y_n]$. By \cite[Theorem 4.3]{CWZ}, $A$ satisfies the
hypothesis in Proposition \ref{xxpro2.4}(1).
By \cite[Theorem 1.3]{EW}, $B$ satisfies
the hypothesis in Proposition \ref{xxpro2.4}(1) (in the semisimple
Hopf case). By \cite[Example 5.5]{RRZ1}, the hypothesis in Proposition
\ref{xxpro2.4}(2) holds. So the assertion follows.
\end{proof}

Sometimes a similar idea applies to a non-tensor product.
Here is an example.

\begin{corollary}
\label{xxcor2.7}
Let $p\in k^\times$ be not a root of unity and
$s\leq n-2$ . Let
$$A=k\langle t_1,t_2,\cdots, t_{n}\rangle/
(t_{n}t_{n-1}-pt_{n-1}t_n-(\sum_{i=1}^s t_i^2), t_i \;
{\text{central for all $i\leq n-2$}}).$$
Then every semisimple Hopf action on $A$ is a group action.
\end{corollary}

\begin{proof} The Nakayama automorphism of $A$ is given in \eqref{E1.5.2}
$$\mu_A: t_{n-1}\to p^{-1} t_{n-1}, \quad t_n\to p t_{n}, \quad t_i\to t_i$$
for all $i\leq n-2$. So $\mu_A$ has eigenvalues $\{1,\cdots, 1, p, p^{-1}\}$.
So $\lambda_i\lambda_j^{-1}$ is not a root of unity when
$i\neq j$ and at least one of $(i,j)$ is either $n$ or $n-1$.
By Lemma \ref{xxlem2.3}(1), $y_{ij}=0$ when
$i\neq j$ and at least one of $(i,j)$ is either $n$ or $n-1$.

Let $H$ be a semisimple Hopf algebra acting on $A$ and $K=H^\circ$.
So $K$ and its Hopf subalgebras are semisimple.

Let $K_1$ be the Hopf subalgebra of $K$ generated by
$y_{ij}$ for all $1\leq i,j\leq n-2$. Thus $K_1$ coacts on
$k[t_1,\cdots,t_{n-2}]$. By \cite[Theorem, 1.3]{EW}, $K_1$ is commutative.
Also $y_{nn}$ and $y_{n-1n-1}$ are grouplike elements. Since
$t_n$ commutes with $t_i$ for all $i\leq n-2$, $y_{nn}$
commutes with $K_1$. Similarly, $y_{n-1n-1}$ commutes
with $K_1$. Applying $\rho$ to the relation
$t_{n}t_{n-1}-pt_{n-1}t_n-(\sum_{i=1}^s t_i^2)$, one sees
that $y_{nn}$ commutes with $y_{n-1n-1}$. Since $K$ is
generated by $K_1$, $y_{nn}$ and $y_{n-1n-1}$, $K$
is commutative. Its dual Hopf algebra $H$ is a group algebra.
\end{proof}

Next we prove a part of Theorems \ref{xxthm0.4} and \ref{xxthm0.5}.
We assume Hypothesis \ref{xxhyp2.1}.

\begin{proposition}
\label{xxpro2.8}
Let $A$ be the algebra $A(1)$, namely, the skew polynomial ring
$k_{p_{ij}}[t_1,t_2,t_3]$, see \eqref{E0.3.1}. Let $H$ act on $A$.
\begin{enumerate}
\item[(1)]
If $p_{12}^{-2} p_{23}p_{31}$, $p_{31}^{-2} p_{12}p_{23}$ and
$p_{23}^{-2} p_{31}p_{12}$ are not roots of unity, then $H$ is
a commutative group algebra.
\item[(2)]
Suppose that $H$ is semisimple and the $H$-action has trivial
homological determinant. If
$p_{12}^{-2} p_{23}p_{31}$, $p_{31}^{-2} p_{12}p_{23}$ and
$p_{23}^{-2} p_{31}p_{12}$ are not 1, then $H$ is a commutative group algebra.
\end{enumerate}
\end{proposition}

\begin{proof} By \eqref{E1.5.1}, the Nakayama automorphism of $A$ is
$$\mu_A: t_1\to \lambda_1 t_1, t_2\to \lambda_2 t_2, t_3\to \lambda_3 t_3$$
where $\lambda_1=p_{21} p_{31}$, $\lambda_2=p_{12}p_{32}$ and $\lambda_3=
p_{13}p_{23}$. So $\lambda_1 \lambda_2^{-1}=p_{12}^{-2} p_{23}p_{31}$,
$\lambda_1^{-1} \lambda_3=p_{31}^{-2} p_{12}p_{23}$ and $\lambda_2 \lambda_3^{-1}=
p_{23}^{-2} p_{31}p_{12}$.

(1) Under the hypothesis of (1), $\lambda_i \lambda_j^{-1}$ are not
roots of unity for all $i\neq j$. The assertion follows
from Lemma \ref{xxlem2.3}(3).

(2) Under the hypothesis of (2), $\lambda_i \lambda_j^{-1}$ are not
1 for all $i\neq j$. The assertion follows from Lemma \ref{xxlem2.3}(6).
\end{proof}

A special case of Proposition \ref{xxpro2.8} is when $(p_{12},p_{13},p_{23})=
(1,1,p)$. If $p$ is not a root of unity, then the hypothesis about
$p_{ij}$ in Proposition \ref{xxpro2.8}(1) holds. So any Hopf action is a
group action. If $p\neq \pm 1, \pm i$, then the hypothesis about
$p_{ij}$ in Proposition \ref{xxpro2.8}(2) holds. So any Hopf action is
a group action (assuming $H$ is semisimple and $H$-action has trivial
homological determinant).

\begin{proposition}
\label{xxpro2.9}
Let $A$ be the algebra $A(2)$, see \eqref{E0.3.2}.
Let $H$ act on $A$.
\begin{enumerate}
\item[(1)]
If $p$ is not a root of unity, then $H$ is a commutative group algebra.
\item[(2)]
Suppose $H$ is semisimple and the $H$-action has trivial
homological determinant. If $p\neq \pm 1$, then $H$ is a commutative
group algebra.
\end{enumerate}
\end{proposition}

\begin{proof} By \eqref{E1.5.2} (taking $n=3$), the Nakayama automorphism of $A$ is
$$\mu_A: t_1\to \lambda_1 t_1, t_2\to \lambda_2 t_2, t_3\to \lambda_3 t_3$$
where $\lambda_1=1$, $\lambda_2=p^{-1}$ and $\lambda_3=p$. So $\lambda_1 \lambda_2^{-1}=
p$,  $\lambda_1^{-1} \lambda_3=p$ and $\lambda_2 \lambda_3^{-1}=p^{-2}$.

(1) Under the hypothesis of (1), $\lambda_i \lambda_j^{-1}$ are not
roots of unity for all $i\neq j$. The assertion follows
from Lemma \ref{xxlem2.3}(3).

(2) Under the hypothesis of (2), $\lambda_i \lambda_j^{-1}$ are not
1 for all $i\neq j$. The assertion follows from Lemma \ref{xxlem2.3}(6).
\end{proof}

The next example shows that if $p$ is a root of unity, there are Hopf algebra
(and non-semisimple Hopf algebra) actions.

\begin{example}
\label{xxex2.10}
There are many (non-group) Hopf actions on $A(2)$ when $p$ is a
root of unity.

Let $T$ be any finite dimensional Hopf algebra
that acts on the skew polynomial ring $k_{p}[x_1,x_2]$
satisfying Hypothesis \ref{xxhyp2.1} and with trivial homological
determinant. Such $T$-actions are
classified in \cite{CKWZ}. In particular, there are many
semisimple Hopf algebras $T$ which are not group algebras
if $p=-1$ and there are many non-semisimple Hopf algebras
$T$ which are not group algebras if $p\neq \pm 1$.

Such an $T$ acts on $A(2)$. Suppose
$$\begin{aligned}
h \cdot x_1& = f_{11}(h) x_1 +f_{12}(h) x_2,\\
h \cdot x_2& = f_{21}(h) x_1 +f_{22}(h) x_2
\end{aligned}
$$
for all $h\in T$ with some $k$-linear maps $f_{ij}: T\to k$. Then
we define
$$\begin{aligned}
h \cdot t_1& =\epsilon(h) t_1,\\
h \cdot t_2& = f_{11}(h) t_2 +f_{12}(h) t_3,\\
h \cdot t_3& = f_{21}(h) t_2 +f_{22}(h) t_3
\end{aligned}
$$
for all $h\in T$. It is easy to check that this is a well defined
$T$-action on $A(2)$ satisfying Hypothesis \ref{xxhyp2.1} with
trivial homological determinant.
\end{example}

Here is a general result dealing with the AS regular algebras generated
by two elements.

\begin{proposition}
\label{xxpro2.10} Let $A$ be a noetherian AS regular algebra
generated by two elements in degree one, say by $x$ and $y$.
Suppose that $\mu_A$ maps $x$ to $\lambda_1 x$ and $y$
to $\lambda_2 y$ such that $\lambda_1^{-1}\lambda_2$
is not a root of unity. Assume further that $A$ has a relation
of the form
$$r:=w_1 m_1 xy m_2+ w_2 m_1 yx m_2+w_3 z+ \cdots=0$$
where $w_1,w_2,w_3\in k^\times$ and $m_1,m_2, z$ are monomials. Assume
that $\{m_1 yx m_2, z, \cdots\}$ are $k$-linearly independent in $A$.
Then any Hopf algebra $H$ which acts on $A$ is a commutative group algebra.
\end{proposition}

\begin{proof} We may assume that $w_1=-1$.
Suppose $\rho: A\to A\otimes K$ is the
correspondence coaction. By Lemma \ref{xxlem2.3}(2), $K$ is
a group algebra and $\rho(x)=x\otimes y_{11}$ and
$\rho(y)=y\otimes y_{22}$
for some group-like elements $y_{11}$ and $y_{22}$
in $K$. Write $m_1=x^{a_{1}}y^{a_2}\cdots y^{a_s}$,
$m_2=x^{b_1}y^{b_2}\cdots y^{b_t}$ and
$z=x^{c_{1}}y^{c_2}\cdots y^{c_u}$ for the monomials appearing
in $r$. Applying $\rho$ to $0=r$, we have
$$\begin{aligned}
0&=\rho(r)=\rho(-m_1 xy m_2+ w_2 m_1 yx m_2+ w_3 z \cdots)\\
&=-m_1xym_2\otimes (y_{11}^{a_1}y_{22}^{a_2}\cdots y_{22}^{a_s} y_{11}y_{22}
y_{11}^{b_1}y_{22}^{b_2}\cdots y_{22}^{b_t})\\
&\qquad
+w_2 m_1yxm_2\otimes (y_{11}^{a_1}y_{22}^{a_2}\cdots y_{22}^{a_s} y_{22}y_{11}
y_{11}^{b_1}y_{22}^{b_2}\cdots y_{22}^{b_t})+\cdots\\
&=(w_2 m_1yxm_2+\cdots)
\otimes (y_{11}^{a_1}y_{22}^{a_2}\cdots y_{22}^{a_s} y_{11}y_{22}
y_{11}^{b_1}y_{22}^{b_2}\cdots y_{22}^{b_t})\\
&\qquad
+w_2 m_1yxm_2\otimes (y_{11}^{a_1}y_{22}^{a_2}\cdots y_{22}^{a_s} y_{22}y_{11}
y_{11}^{b_1}y_{22}^{b_2}\cdots y_{22}^{b_t})+\cdots\\
&= w_2 m_1yxm_2\otimes (y_{11}^{a_1}y_{22}^{a_2}\cdots y_{22}^{a_s}
(y_{22}y_{11}-y_{11}y_{22})
y_{11}^{b_1}y_{22}^{b_2}\cdots y_{22}^{b_t})+\cdots
\end{aligned}
$$
which implies that $y_{11}y_{22}=y_{22}y_{11}$. So $K$ is commutative
and cocommutative. Since we assume that $k$ is algebraically
closed of characteristic zero, $K$ (and then $H$) is a commutative
group algebra.
\end{proof}

\begin{remark}
\label{xxrem2.11}
By the proof of Proposition \ref{xxpro2.10}, the hypothesis of
``$\lambda_i \lambda_j^{-1}$ not being a root of unity'' can be
replaced by the hypothesis ``that $K$ is a group algebra with
$\rho(x)=x\otimes y_{11}$ and $\rho(y)=y\otimes y_{22}$''.
\end{remark}

\begin{proposition}
\label{xxpro2.12}
Let $A$ be the down-up algebra algebra $A(6)=A(\alpha, \beta)$, see \eqref{E0.3.6}.
Let $H$ act on $A$ and $K=H^\circ$.
\begin{enumerate}
\item[(1)]
If $\beta$ is not a root of unity, then $K$ is a
group algebra.
\item[(2)]
If $\alpha\neq 0$ and $\beta$ is not a root of unity, then $H$ is a commutative
group algebra.
\item[(3)]
Suppose $H$ is semisimple and the $H$-action has trivial
homological determinant. If $\beta\neq \pm 1$, then $K$ is a
group algebra. If further $\alpha\neq 0$, then $H$ is a commutative
group algebra.
\end{enumerate}
\end{proposition}

\begin{proof} By \eqref{E1.5.6},  the Nakayama automorphism of $A$ is
$$\mu_A: x\to -\beta x, y\to -\beta^{-1} y$$
So $\lambda_1 \lambda_2^{-1}=\beta^2$.

(1,2) Under the hypothesis of (1), $\lambda_1 \lambda_2^{-1}$ is not
a root of unity.  By Lemma \ref{xxlem2.3}(2), $K$ is a group algebra.
So part (1) follows. In part (2), we further assume that $\alpha\neq 0$,
then Proposition \ref{xxpro2.10} applies. So $H$ is a commutative
group algebra.

(3) Under the hypothesis of (3), $\lambda_1 \lambda_2^{-1}$ is not
$1$. By Lemma \ref{xxlem2.3}(5), $K$ is a group algebra.
When $\alpha\neq 0$, the assertion follows from Remark \ref{xxrem2.11}.
\end{proof}

We will state a few more results. The proofs are very similar to the
proof of Proposition \ref{xxpro2.12}, so we decide to omit them.

\begin{proposition}
\label{xxpro2.13}
Let $A$ be the algebra $A(7)=S(p)$, see \eqref{E0.3.7}.
Let $H$ act on $A$ and $K=H^\circ$.
\begin{enumerate}
\item[(1)]
If $p$ is not a root of unity, then $K$ is a group algebra.
\item[(2)]
Suppose $H$ is semisimple and the $H$-action has trivial
homological determinant. If $p\neq \pm i$, then $K$ is a
group algebra.
\end{enumerate}
\end{proposition}

\begin{proposition}
\label{xxpro2.14}
Let $A$ be the algebra ${\bf D}$ with parameter $(p,q)$, see the end of
Section \ref{xxsec1}. Let $H$ act on $A$.
\begin{enumerate}
\item[(1)]
If $p^{-3} q^4$ is not a root of unity, then $H$ is a commutative
group algebra.
\item[(2)]
Suppose $H$ is semisimple and the $H$-action has trivial
homological determinant. If $p^{-3}q^{4}\neq \pm 1$, then $H$ is a commutative
group algebra.
\end{enumerate}
\end{proposition}

\begin{proposition}
\label{xxpro2.15}
Let $A$ be the algebra ${\bf G}$, see the end of
Section \ref{xxsec1}. Let $H$ act on $A$.
\begin{enumerate}
\item[(1)]
If $g$ is not a root of unity, then $H$ is a commutative
group algebra.
\item[(2)]
Suppose $H$ is semisimple and the $H$-action has trivial
homological determinant. If $q\neq \pm 1$, then $H$ is a commutative
group algebra.
\end{enumerate}
\end{proposition}

It is expected that the above theorem holds for the class of
AS regular algebras of global dimension four with two generators
and generic parameters.

Finally we classify all possible finite dimensional Hopf algebra
actions on $A(0,\beta)$ and $S(p)$ that satisfies
Hypothesis \ref{xxhyp2.1}.

\begin{proposition}
\label{xxpro2.16}
Let $A$ be $A(0,\beta)$ or $S(p)$ where $\beta$ and
$p$ are not roots of unity. Consider the Hopf algebra actions on $A$
satisfying Hypothesis {\rm{\ref{xxhyp2.1}}}. Then $H$ acts on $A$ if and
only if the dual Hopf algebra $K=H^\circ$ is a group algebra $kG$ where
$G$ is a quotient group of $\langle a,b \mid a^2 b=ba^2, ab^2=b^2 a\rangle$.
\end{proposition}

\begin{proof}
If $H$ acts on $A$, then we have a right $K$-coaction $\rho: A\to A\otimes K$ with
$\rho(x)=x\otimes a_{11} + y\otimes a_{21}$ and $\rho(y)=x\otimes a_{12}+y\otimes a_{22}$, for some $a_{11},a_{12},a_{21},a_{22}\in K$.
By \ref{E1.5.6} and \ref{E1.5.7}, we have
$$\mu_{A}: \begin{cases}
x\to -\beta x, \quad y\to -\beta^{-1}y,\,\, \text{if} \,\, A=A(0,\beta) \\
x\to  p x, \quad y\to -p^{-1}y, \,\,\quad \text{if} \,\,\, A=S(p).
\end{cases} $$
Hence $\lambda_1\lambda_2^{-1}=\begin{cases}
\beta^2, \,\,\,\,\,\,\quad \text{if} \,\, A=A(0,\beta) \\ -p^2, \,\,\quad \text{if} \,\,\, A=S(p)
\end{cases}$ is not a root of unity. By Lemma \ref{xxlem2.3}, 
$a_{12}=a_{21}=0$. Thus $\rho(x)=x\otimes a_{11}$ and 
$\rho(y)=y\otimes a_{22}$.
Since the $K$-coaction on $A$ is inner faithful, $\{a_{11},a_{22}\}$ 
generates $K$ as a Hopf algebra. We have 
$\Delta_K(a_{ii})=a_{ii}\otimes a_{ii}$ by the 
coassociativity of $\rho$. Applying $\rho$ to 
the equations $x^2y=\beta yx^2, xy^2=\beta y^2x$ when  $A=A(0,\beta)$, 
we obtain that
\begin{align*}
x^2y\otimes a_{11}^2a_{22}=\rho(x^2y)=\beta\rho(yx^2)
=\beta yx^2\otimes a_{22}a_{11}^2=x^2y\otimes a_{22}a_{11}^2\\
xy^2\otimes a_{11}a_{22}^2=\rho(xy^2)=\beta\rho(y^2x)
=\beta y^2x\otimes a_{22}^2a_{11}=xy^2\otimes a_{22}^2a_{11}.
\end{align*}
Hence $a_{11}^2a_{22}=a_{22}a_{11}^2$ and $a_{11}a_{22}^2=a_{22}^2a_{11}$. We can similarly get $a_{11}^2a_{22}=a_{22}a_{11}^2$ 
and $a_{11}a_{22}^2=a_{22}^2a_{11}$ when $A=S(p)$ by applying $\rho$
to the equations $x^2y=pyx^2$ and  $xy^2=-py^2x$. 
Therefore, $K$ is a group algebra $kG$ where
$G$ is a quotient group of $\langle a,b \mid a^2 b=ba^2, ab^2=b^2 a\rangle$.

Conversely, we assume that $K$ is a group algebra $kG$ where
$G$ is a quotient group of $\langle a,b \mid a^2 b=ba^2, ab^2=b^2 a\rangle$.
Define a right $K$-coaction $\rho: A\to A\otimes K$ by 
$\rho(x)=x\otimes a$ and $\rho(y)=y\otimes b$. It is easy to check 
that the corresponding Hopf action on $A$ by $H=K^{\circ}$ satisfies 
the Hypothesis {\rm{\ref{xxhyp2.1}}}.
\end{proof}

As a consequence, if $A$ is either $A(\alpha,\beta)$ or $S(p)$ with generic
parameters, then $A^H$ is not AS regular for all non-trivial $H$.

\section{Hopf actions on $A$ with non-diagonalizable $\mu_A$}
\label{xxsec3}

In this section we will prove Theorems \ref{xxthm0.4} and \ref{xxthm0.5}
for algebras $A(3), A(4), A(5)$. We start with a few lemmas concerning
the automorphism $\eta_{\mu_A^{\tau}}$ of $K$.

\begin{lemma}
\label{xxlem3.1} Suppose $K$ is a finite dimensional Hopf algebra coacting
on $A$ via $\rho$. Let $y_{11}$ be a grouplike element in $K$.
\begin{enumerate}
\item[(1)]
If $z_{12}\in K$ satisfies $\Delta(z_{12})=y_{11}\otimes z_{12}+z_{12}\otimes
y_{11}$, then $z_{12}=0$.
\item[(2)]
Suppose $\{v,w\}$ are linearly independent in $A$ and $\rho(v)=v\otimes y_{11}$
and $\rho(w)=w\otimes y_{11}+v\otimes z_{12}$. Then $z_{12}=0$.
\end{enumerate}
\end{lemma}

\begin{proof} (1) Let $x=y_{11}^{-1}z_{12}$, then $x$ is a primitive element in $K$.
Since ${\text{char}}\; k=0$ and $K$ is finite dimensional, $K$ does not contain
any nonzero primitive element. So $x=0$ and whence $z_{12}=0$.

(2) Applying coassociativity, we have $\Delta(z_{12})=y_{11}\otimes z_{12}+z_{12}\otimes
y_{11}$. Then assertion follows from part (1).
\end{proof}

In the next lemma the global dimension of $A$ could be larger than $2$.
We say an $H$-action on $A$ is graded trivial if $H=k\langle \sigma\rangle$
and $\sigma$ is a graded algebra automorphism of $A$ of the form
$\sigma: a\to \xi^{\deg a} a$ for some root of unity $\xi\in k$.

\begin{proposition}
\label{xxpro3.2}
Suppose $A$ is an AS regular algebra generated by $t_1$ and $t_2$ in degree one
such that $\mu_A$ is non-diagonalizable. Let $H$ be a Hopf algebra acting on $A$.
Then $H$ is isomorphic to $k {\mathbb Z}/(n)$ for some $n$ and the
$H$-action on $A$ is graded trivial.
\end{proposition}

\begin{proof} Up to a base change we may assume that
$\mu_A$ sends $t_1\to at_1$ and $t_2\to at_2+at_1$. So the matrix
${\mathbb M}$ defined by \eqref{E2.2.2} is $\begin{pmatrix} a& 0\\a&a\end{pmatrix}$.
Let $\rho: A\to A\otimes K$ be the coaction induced by the $H$-action and write
$\rho(t_i)=\sum_{s=1}^2 t_s\otimes y_{si}$ as defined in \eqref{E2.2.3}.
Then the map $f:=\eta_{\mu_A^\tau}$ defined in \eqref{E2.2.4} sends
$$\begin{aligned}
y_{11}&\to y_{11}+y_{21},\\
y_{12}&\to -(y_{11}+y_{21})+y_{12}+y_{22},\\
y_{21}&\to y_{21},\\
y_{22}&\to -y_{21}+y_{22}.
\end{aligned}
$$
It is easy to see that $f^n(y_{11})=y_{11}+n y_{21}$. Since $f$ has finite order and
${\text{char}}\; k=0$, $y_{21}=0$. As a consequence, $f(y_{ii})=y_{ii}$ for $i=1,2$.
Now we have $f^n(y_{12})=y_{12}+n(y_{22}-y_{11})$. Since $f$ has finite order and
${\text{char}}\; k=0$, $y_{22}-y_{11}=0$. Since $y_{22}=y_{11}$, $\Delta(y_{12})=y_{11}
\otimes y_{12}+y_{12}\otimes y_{11}$. By Lemma \ref{xxlem3.1}(1), $y_{12}=0$. Thus
$K=k\langle y_{11}\rangle$ and $H=k\langle \sigma \rangle $ where $\sigma$
maps $t_i$ to $\xi t_i$ for some root of unity. Thus the $H$-action on $A$
is graded trivial.
\end{proof}

Again in the next lemmas, $A$ is AS regular of finite global dimension.

\begin{lemma}
\label{xxlem3.3} Suppose $A$ is generated by elements $t_1,t_2,t_3$
and $K$-coacts on $A$ with $\rho(t_i)=\sum_{s=1}^3 t_j \otimes y_{ji}$
for all $i=1,2,3$, where $\{y_{ij}\}_{1\leq i,j\leq 3}$ are elements in
$K$.

\begin{enumerate}
\item[(1)]
If $\mu_A: t_1\to a t_1, t_2\to a t_2 +b t_1, t_3\to a^{-2} t_3$ for some
$a,b\in k^\times$, then $\eta_{\mu_A^{\tau}}$ is determined by
$$\eta_{\mu_A^{\tau}}: \quad
\begin{aligned}
y_{11} & \to y_{11}+a^{-1}b y_{21}\\
y_{12} & \to -a^{-1}b y_{11}-a^{-2}b^2y_{21}+y_{12}+a^{-1}by_{22}\\
y_{13} & \to a^{3} y_{13}+a^2 by_{23}\\
y_{21} & \to y_{21}\\
y_{22} & \to -a^{-1}by_{21}+y_{22}\\
y_{23} & \to a^3 y_{23}\\
y_{31} & \to a^{-3}y_{31}\\
y_{32} & \to -a^{-4}b y_{31}+a^{-3}y_{32}\\
y_{33} & \to y_{33}\\
\end{aligned}
$$
\item[(2)]
If $\mu_A: t_1\to t_1, t_2\to t_2 +b t_1, t_3\to t_3+ bt_2+b^2 t_1$ for some
$b\in k^\times$, then $\eta_{\mu_A^{\tau}}$ is determined by
$$\eta_{\mu_A^{\tau}}: \quad
\begin{aligned}
y_{11} & \to y_{11}+ay_{21}+a^2y_{31}\\
y_{12} & \to y_{12}-ay_{11}+ay_{22}-a^2y_{21}+a^2y_{32}-a^3y_{31}\\
y_{13} & \to y_{13}-ay_{12}+ay_{23}-a^2y_{22}+a^2y_{33}-a^3y_{32}\\
y_{21} & \to y_{21}+ay_{31}\\
y_{22} & \to y_{22}-ay_{21}+ay_{32}-a^2y_{31}\\
y_{23} & \to y_{23}-ay_{22}+ay_{33}-a^2y_{32}\\
y_{31} & \to y_{31}\\
y_{32} & \to y_{32}-ay_{31}\\
y_{33} & \to y_{33}-ay_{32}\\
\end{aligned}
$$
\end{enumerate}
\end{lemma}

\begin{proof} If necessary, please go back to Section \ref{xxsec2} to review
the definition of $\eta_{\mu_A^{\tau}}$.

(1) In this case, ${\mathbb M}
=\begin{pmatrix} a& 0& 0\\ b &a &0\\0&0& a^{-2}\end{pmatrix}$.
By \eqref{E2.2.4}, $\eta_{\mu_A^{\tau}}(y_{ij})={\mathbb M}^{\tau} (y_{ij})
({\mathbb M}^{\tau})^{-1}$. The assertion follows by a direct computation.

(2) Similar to (1).
\end{proof}

\begin{lemma}
\label{xxlem3.4} Retain the hypotheses of Lemma {\rm{\ref{xxlem3.3}}}.
\begin{enumerate}
\item[(1)]
If $\eta_{\mu_A^{\tau}}$ in Lemma {\rm{\ref{xxlem3.3}(1)}} has finite order,
then $y_{21}=y_{23}=y_{31}=y_{11}-y_{22}=0$ and $y_{11}$ and $y_{33}$ are grouplike elements.
If $a$ is not a root of unity, then $y_{ij}=0$ for all $i\neq j$.
\item[(2)]
If $\eta_{\mu_A^{\tau}}$ in Lemma {\rm{\ref{xxlem3.3}(1)}} is the identity
and $a^3\neq 1$, then $y_{ij}=0$ for all $i\neq j$.
\item[(3)]
If $\eta_{\mu_A^{\tau}}$ in Lemma {\rm{\ref{xxlem3.3}(2)}} has finite order,
then $y_{ij}=0$ and $y_{ii}=y_{jj}$ for all $i\neq j$.
\end{enumerate}
\end{lemma}

\begin{proof} Let $f=\eta_{\mu_A^{\tau}}$ in the proof.

(1) Applying $f$ to $y_{11}$ multiple times, we have $f^n(y_{11})=y_{11}+na^{-1}by_{21}$.
Since $f$ has finite order (and ${\text{char}}\; k=0$),
$y_{21}=0$. As a consequence, $f(y_{11})=y_{11}$ and $f(y_{22})=y_{22}$.
Applying $f$ to $y_{12}$, one sees that $f^n(y_{12})=y_{12}+na^{-1}b(y_{22}-y_{11})$.
Hence $y_{11}=y_{22}$.

Now we have $f(y_{13})=a^3 y_{13}+a^2b y_{23}$ and
$f(y_{23})=a^3 y_{23}$. So induction shows that $f^n(y_{13})=a^{3n} y_{13}+(n-1)a^{3n-1}by_{23}$
for all $n$.
Suppose $y_{13}\neq 0$. Then 
the fact that $f$ has finite order implies that
$y_{23}=\alpha y_{13}$. Then equations $f(y_{13})=a^3 y_{13}+a^2b y_{23}$ and
$f(y_{23})=a^3 y_{23}$ imply that $y_{23}=0$. If $y_{13}=0$, then
the equation $f(y_{13})=a^3 y_{13}+a^2b y_{23}$ implies that $y_{23}=0$.
Therefore $y_{23}=0$. Similarly, one can show that $y_{31}=0$.

Using the equations $y_{21}=y_{23}=y_{31}=0$, one can easily check that
$y_{11}=y_{22}$ and $y_{33}$ are grouplike elements.

Finally, if $a$ is not a root of unity, then the finite-orderness of $f$
implies that $y_{13}=y_{32}=0$. In this case, $\Delta(y_{12})=y_{11}\otimes y_{11}
+y_{11}\otimes y_{12}$. It follows from Lemma \ref{xxlem3.1}(1) that $y_{12}=0$.
Thus we complete the proof of $y_{ij}=0$ for all $i\neq j$.

(2) The proof is similar to the proof of part (1).

(3) Note that $f^n(y_{32})=y_{32}-na y_{31}$ for all $n$ and $na\neq 0$ for
all $n>0$. Since $f$ has finite order, $y_{31}=0$. Similarly, $y_{32}=0$.
Then $f(y_{21})=y_{21}$. Applying $f$ to $y_{11}$, one sees that $y_{21}=0$.
Now $f(y_{ii})=y_{ii}$ for all $i$.
Since $y_{32}=0$, $f(y_{23})=y_{23}+a(y_{33}-y_{22})$. Since $f$ has finite
order, we have $y_{33}-y_{22}=0$. Applying $f$ to $y_{12}$, one sees that
$y_{22}=y_{11}$. Now applying $f$ to $y_{22}$, one sees that $y_{31}=0$.
Finally applying $f$ to $y_{13}$, one sees that $y_{12}=y_{23}$. Since
$y_{32}=0$, we have $\Delta(y_{12})=y_{11}\otimes y_{12}+y_{12}\otimes y_{11}$.
By Lemma \ref{xxlem3.1}(1), $y_{12}=0$ (and whence $y_{23}=0$). Finally,
$$\Delta(y_{13})=y_{11}\otimes y_{13}+y_{12}\otimes y_{23}+y_{13}\otimes y_{33}
=y_{11}\otimes y_{13}+y_{13}\otimes y_{11}$$
which implies that $y_{13}=0$ by Lemma \ref{xxlem3.1}(1).
\end{proof}

Going back to global dimension three we have the following.

\begin{proposition}
\label{xxpro3.5}
Let $A$ be the algebra $A(3)$ or $A(4)$ and let $H$ act on $A$.
\begin{enumerate}
\item[(1)]
If $q$ is not a root of unity  when $A=A(3)$ and $p$ is not a root of unity
when $A=A(4)$, then $H$ is a commutative group algebra.
\item[(2)]
If $H$ is semisimple, then $H$ is a commutative group algebra.
\end{enumerate}
\end{proposition}

\begin{proof} (1) By \eqref{E1.5.3} and \eqref{E1.5.4}, the
Nakayama automorphism of $A$ is of the form given in Lemma \ref{xxlem3.3}(1).
By Lemma \ref{xxlem3.4}(1), $y_{ij}=0$ for all $i\neq j$. So the $K$-coaction
is given by $\rho(t_1)=t_1\otimes y_{11}, \rho(t_2)=t_2\otimes y_{11}$ and
$\rho(t_3)=t_3\otimes y_{33}$ where each $y_{ii}$ is a grouplike element. 
So $K$ is a group
algebra. Since $t_1$ and $t_3$ are skew commutative, $y_{11}$ commutes
with $y_{33}$, see the proof of Proposition \ref{xxpro2.10}. So $K$
(and whence $H$) is a commutative group algebra.

(2) Assume $H$ is semisimple. Then $A_1$ is a direct sum of simples.
Let $C$ be the sub-coalgebra generated by $\{y_{ij}\}_{1\leq i,j\leq 3}$.
By Lemma \ref{xxlem3.4}(1), $y_{21}=y_{23}=y_{31}=y_{11}-y_{22}=0$.
Then every simple comodule is $1$-dimensional. Therefore $A_1$ is a direct
sum of three 1-dimensional simple $K$-comodules, or three
1-dimensional simple $H$-modules. Since $H$-action is inner-faithful,
$H$ is commutative, or $K$ is cocommutative. Applying $\Delta$ to $y_{13}$
and using the cocommutativity, we have $y_{33}=y_{11}$ or $y_{13}=0$.
If $y_{11}=y_{33}$, then Lemma \ref{xxlem3.1}(1) implies that
$y_{13}=0$. Therefore $y_{13}=0$. Similarly, $y_{31}=y_{32}=0$,
so $y_{ij}=0$ for all $i\neq j$. The rest of the proof is similar to
the proof of part (1).
\end{proof}

Up to a $k$-linear basis change, the Nakayama automorphism of
$A(5)$ is of the form given in Lemma \ref{xxlem3.3}(2).
By using Lemma \ref{xxlem3.4}(3) and the ideas in the proof of
Proposition \ref{xxpro3.5}, we have the following.

\begin{proposition}
\label{xxpro3.6}
Let $A$ be the algebra $A(5)$ and let $H$ act on $A$.
Then $H$ is a commutative group algebra and the $H$-action is graded
trivial.
\end{proposition}

Now we are ready to prove Theorems \ref{xxthm0.4} and \ref{xxthm0.5}.

\begin{proof}[Proof of Theorem \ref{xxthm0.4}]
(1) This is Proposition \ref{xxpro2.8}(1).

(2) Proposition \ref{xxpro2.9}(1).

(3,4) Proposition \ref{xxpro3.5}(1).

(5) Proposition \ref{xxpro3.6}.

(6) Proposition \ref{xxpro2.12}(2).
\end{proof}

\begin{proof}[Proof of Theorem \ref{xxthm0.5}]

(1) This is Proposition \ref{xxpro2.8}(2).

(2) Proposition \ref{xxpro2.9}(2).

(3,4) Proposition \ref{xxpro3.5}(2).

(6) Proposition \ref{xxpro2.12}(3).

(7) Proposition \ref{xxpro2.13}(2).
\end{proof}

Statements in (0.4.8), (0.4.9), (0.5.8) and (0.5.9) are
proved in Propositions \ref{xxpro2.14} and \ref{xxpro2.15}.

To conclude this section we prove a lemma which should be
useful for the study of Hopf actions on AS regular algebras of higher
global dimension. Suppose $K$ coacts on an AS regular algebra $A$ and
$A$ is generated by $V:=A_1$. Suppose that $V$ is a direct sum of simple
$K$-comodules $\oplus_{i=1}^\alpha V_i$. We can choose a basis for each
$V_i$, say $v^i_{1}, v^i_{2},\cdots, v^i_{n_i}$, for each $i$. Then
$K$-coaction on $V$ is determined by
$$\rho(v^i_{t})=\sum_{s=1}^{n_i} v^i_{s}\otimes y^i_{st}$$
for all $t=1,\cdots,n_i$. For a fixed $i$, $\{y^i_{st}\}_{1\leq s,t\leq n_i}$
is a matrix coalgebra of dimension $n_i^2$. Let $Y_i$ be the matrix
$(y^i_{st})_{n_i\times n_i}$, and $Y$ the block matrix $diag(Y_1,Y_2,
\cdots, Y_{\alpha})$. Write the Nakayama automorphism
of $A$ in terms of the basis $\{v^i_j\}$,
$$\mu_A(v^i_j)=\sum_{i,j} m^{ii'}_{jj'} v^{i'}_{j'}.$$

\begin{lemma}
\label{xxlem3.7}
Retain the above notation.
Let $N$ be the order of the automorphism $\eta_{\mu_A^\tau}$.
\begin{enumerate}
\item[(1)]
The matrix ${\mathbb M}^N$ is of the form
$diag (r_1 I_{n_1},r_2 I_{n_2},\cdots, r_\alpha I_{\alpha})$
for some $r_i\in k^\times$.
\item[(2)]
${\mathbb M}$ is diagonalizable.
\item[(3)]
If $V$ is a simple $H$-module, or a simple $K$-comodule, then  ${\mathbb M}^N$
is $r I_{n}$ for some $r\in k^\times$. In this case, the order of $\mu_A$
divides $2l\dim_k K$ where $l$ is the AS index of $A$.
\item[(4)]
If $V$ is a direct sum of 1-dimensional $H$-modules, then
$H$ is commutative.
\end{enumerate}
\end{lemma}

\begin{proof} (1) Since $N$ be the order of the automorphism
$\eta_{\mu_A^\tau}$, ${\mathbb M}^N$ commutes with $Y$.

The assertion follows from the following fact from linear algebra:
if $X$ is a matrix in $M_{n}(k)$ such that
$X Y=YX$, then $X=diag (r_1 I_{n_1},r_2 I_{n_2},\cdots, r_\alpha I_{\alpha})$,
for some $r_i\in k^\times$.

(2) Since $diag (r_1 I_{n_1},r_2 I_{n_2},\cdots, r_\alpha I_{\alpha})$ is
diagonalizable, so is ${\mathbb M}$.

(3) The first assertion follows from part (1) by taking $\alpha=1$.
Let $g$ be the automorphism $\mu_A^N$. Then $g$ is graded trivial
sending $a\to r^{\deg a} a$ for all $a\in A$. So $\hdet g= r^l$
\cite[Lemma 5.3(a)]{RRZ1}. By \cite[Theorem 5.3]{RRZ2}, $\hdet \mu_A=1$,
so $\hdet g=1$. Thus $r^l=1$. Combining with the fact that $N$ divides
$2l\dim_k K$, we obtain the assertion.

(4) Since $A_{-1}$ is a direct sum of 1-dimensional
$H$-module, $H/[H,H]$ acts on $A_{-1}$. Since the $H$-action is inner
faithful, $[H,H]=0$ or $H$ is commutative.
\end{proof}

\section{Nakayama automorphism commutes with all automorphisms}
\label{xxsec4}

We first briefly recall the definition of Hochschild cohomology.
Let $M$ be an $A$-bimodule, or equivalently, a left $A^e$-module
where $A^e:=A\otimes A^{op}$. Consider the cochain complex
\begin{equation}
\label{E4.0.1}\tag{E4.0.1}
\qquad\qquad\qquad\qquad\qquad\qquad\qquad\qquad
\qquad\qquad\qquad\qquad
\end{equation}
$$C_A(M):
0\to M\xrightarrow{d_1}\Hom_k(A,M)\xrightarrow{d_2} \Hom_k(A^{\otimes 2},M)
\xrightarrow{d_3}\Hom_k(A^{\otimes 3},M)\to \cdots,$$
where $d_n=\sum_{i=0}^{n} (-1)^i\partial_i$, and for any
$f\in \Hom_k(A^{\otimes n-1}, M)$,
$$(\partial_i)(f)(a_1,\cdots,a_n)=\begin{cases}
a_1 f(a_2,\cdots, a_n), & {\text{if}} \quad i=0,\\
f(a_1,\cdots, a_{i}a_{i+1}, \cdots, a_n) & {\text{if}} \quad 1\le i\le n-1,\\
f(a_1,\cdots, a_{n-1})a_{n}, & {\text{if}} \quad i=n.\\
\end{cases}$$

The $n$-th Hochschild cohomology of $A$ with coefficients in $M$ is
$$H^n(A,M)={\rm{H}}^n (C_A(M)).$$
As $A$ is free over $k$, one also has that $H^n(A,M)\cong
\Ext^n_{A^e}(A,M)$. A special case is when 
$M=A^e$ as an $A^e$-bimodule (and also an $A$-bimodule). Then
$H^n(A,M)$ is a right $A^e$-module, which can also be viewed
naturally as $A$-bimodules. We need the following lemma.

\begin{lemma}
\label{xxlem4.1}
Let $\sigma: A\to B$ be an isomorphism of algebras. Then $\sigma$
 induces a $k$-linear isomorphism $\Phi^n: H^n(A,A\otimes A^{op})\to
H^n(B,B\otimes B^{op})$ such that $\Phi^n( a_1 f a_2)=\sigma(a_1)
\Phi^n(f)\sigma(a_2)$ for all $a_1,a_2\in A$.\end{lemma}

\begin{proof}
 In fact we construct a map at the level of complexes
$\Psi: C_A(A\otimes A^{op})
\to C_B(B\otimes B^{op})$ such that $\Psi^n( a_1 f a_2)=\sigma(a_1)
\Psi^n(f)\sigma(a_2)$ for all $a_1,a_2\in A$. Then the assertion follows
from taking homology. For any $f\in \Hom_k(A^{\otimes n}, A\otimes A^{op})$,
define $\Psi(f)\in \Hom_k(B^{\otimes n},B\otimes B^{op})$ by
$$\Psi(f)(b_1,\cdots, b_n)=(\sigma\otimes \sigma)(f(\sigma^{-1}(b_1),
\cdots,\sigma^{-1}(b_n))$$
where $\sigma\otimes \sigma$ is the corresponding isomorphism
from $A\otimes A^{op}\to B\otimes B^{op}$.
One can easily show that $\Psi$ commutes with $\partial_i$ for all
$i$. Hence $\Psi$ commutes with the differential. It also easily
follows from the definition that
$\Psi(a_1 f a_2)=\sigma(a_1)\Psi(f)\sigma(a_2)$
for all $a_1,a_2\in A$. The assertion follows by taking
homology.
\end{proof}

As proved in \cite[Lemma 1.2]{RRZ1}, every Artin-Schelter regular
algebra and every Artin-Schelter regular noetherian Hopf algebra
has a Nakayama automorphism, see \cite{BrZ} and \cite[Lemma 1.3]{RRZ1}.
We  now show Theorem \ref{xxthm0.6}. The group of all algebra
automorphisms of $A$ is denoted by $\Aut(A)$.

\begin{theorem}
\label{xxthm4.2}
Let $A$ be an algebra with Nakayama automorphism $\mu$. Let $g\in
\Aut(A)$. Then $g\mu g^{-1}$ is also a Nakayama automorphism of $A$.
If $A^\times$ is in the center of $A$ {\rm{(}}namely, $A$ has no non-trivial
inner automorphisms{\rm{)}}, then $\mu$ commutes with every $g\in \Aut(A)$.
\end{theorem}

\begin{proof} Let $d$ and $\mu$ be defined as in \eqref{E1.1.1}. 
Consider $B=A$ and let $\sigma: A\to B=A$ be the map $g$.
By Lemma \ref{xxlem4.1},
$\Phi^d: H^d(A,A\otimes A^{op})\to H^d(A,A\otimes A^{op})$ is a $k$-linear
automorphism such that $\Phi^d(a_1 f a_2)=\sigma(a_1) \Phi^d(f) \sigma (a_2)$
for all $a_1,a_2\in A$. By definition $H^d(A,A\otimes A^{op})\cong A^{\mu}$
with a generator $u$. Then $u':=\Phi^d(u)$ is also an $A$-bimodule
generator. Hence $u'=c u$ for some $c\in A^\times $. Applying $\Phi^d$ to
the equation $\mu(a)u =ua$, we obtain that $g(\mu(a)) u'=u' g(a)$,
or equivalently, $g(\mu(g^{-1}(a))) u'=u' a$ for all $a\in A$. Thus
$g\mu g^{-1}$ is a Nakayama automorphism of $A$.
\end{proof}

\begin{lemma}
\label{xxlem4.3} Let $A$ be an algebra with Nakayama automorphism
$\mu_A$ and such that $A^\times =k^\times$.
Suppose that $\{x_1,\cdots,x_n\}$ is a set of generators of $A$ such that
$\mu_A(x_i)=\lambda_i x_i$ for all $i$ and that the
set of the ordered monomials
$\{x_1^{a_1}\cdots x_n^{a_n}\mid a_1,\cdots,a_n\geq 0\}$
spans the whole algebra $A$.
Assume that $\lambda_1$ cannot be written as $\prod_{j>1} \lambda_j^{b_j}$
for any $b_j\geq 0$.  Then, for any algebra automorphism $g$ of $A$, there
is a $c\in k^\times$ such that $g(x_1)=c x_1$.
\end{lemma}

\begin{proof} Suppose $g(x_1)=\sum c_{a} x_1^{a_1}x_{2}^{a_2}\cdots x_n^{a_n}$
and write it as
$$g(x_1)=\sum_{a_1>0} c_{a^*} x_1^{a_1}x_{2}^{a_2}\cdots x_n^{a_n}
+\sum_{b_1=0} c_{b^*} x_2^{b_2}\cdots x_n^{b_n}$$
where monomials $\{x_1^{a_1}x_{2}^{a_2}\cdots x_n^{a_n}\}_{a^*}
\cup \{x_2^{b_2}\cdots x_n^{b_n}\}_{b^*}$ appeared in the above expression with
nonzero $c_{a^*}$ and $c_{b^*}$ are distinct and linearly
independent. Since $\mu_A g \mu_A^{-1}=g$, we have
$$\begin{aligned}
\lambda_1^{-1}
(\sum_{a_1>0} c_{a^*} (\prod_{j} \lambda_j^{a_i}) &
x_1^{a_1}x_{2}^{a_2}\cdots x_n^{a_n}
+\sum_{b_1=0} c_{b^*} (\prod_{j>1} \lambda_j^{b_i})
x_2^{b_2}\cdots x_n^{b_n})\\
&=
\sum_{a_1>0} c_{a^*} x_1^{a_1}x_{2}^{a_2}\cdots x_n^{a_n}
+\sum_{b_1=0} c_{b^*} x_2^{b_2}\cdots x_n^{b_n}.
\end{aligned}
$$
Since $\lambda_{1}\neq \prod_{j>1} \lambda_j^{b_i}$,
all $c_{b^*}=0$. Thus $g(x_1)=x_1 h$ for some $h\in A$.
Similarly, $g^{-1}(x_1)=x_1 f$ for some $f\in A$. Thus
$h,f\in A^{\times}$. Since $A^\times =k^\times$,
$g(x_1)=c x_1$ for some $c\in k^{\times}$.
\end{proof}

The following proposition is well known. For completeness we give a
short proof using Lemma \ref{xxlem4.3}.

\begin{proposition}
\label{xxpro4.4}
Let $A$ be $A(1)$ where  $p_{ij}$ are generic. Then every algebra automorphism
of $A$ preserves the grading of $A$ and $\Aut(A)=(k^\times )^3$.
\end{proposition}

\begin{proof} Let $g$ be an algebra automorphism of $A$.

By \eqref{E1.5.1}, the Nakayama automorphism of $A$ is of the form
$$\mu_A: t_1\to \lambda_1 t_1, t_2\to \lambda_2 t_2, t_3\to \lambda_3 t_3$$
where $\lambda_1=p_{21}p_{31}$, $\lambda_2=p_{12}p_{32}$ and $\lambda_3=p_{13}p_{23}$.
Since $p_{ij}$'s are generic, $\lambda_1$ cannot be written as $\lambda_2^{n_2}\lambda_3^{n_3}$.
Applying Lemma \ref{xxlem4.3} to $(x_1,x_2,x_3)=(t_1,t_2,t_3)$,
$g(t_1)=c_1 t_1$. By symmetry, $g(t_2)=c_2 t_2$ and
$g(t_3)=c_3 t_3$ for some $c_i\in k^\times$. The assertion follows.
\end{proof}

\begin{proposition}
\label{xxpro4.5}
Let $A$ be $A(2)$ where  $p$ is not a root of unity. Then every algebra automorphism
of $A$ preserves the grading of $A$ and $\Aut(A)=(k^\times )^2$.
\end{proposition}

\begin{proof}
By \eqref{E1.5.2}, the Nakayama automorphism of $A$ is given by
$$\mu: t_1\to t_1, \quad t_2\to p^{-1} t_2, \quad
t_3\to p t_3.$$
Let $g$ be any algebra automorphism of $A$. Applying
Lemma \ref{xxlem4.3} to $\{x_1,x_2,x_3\}=\{t_2,t_1,t_3\}$,
we have that $g(t_2)=c t_2$ for some $c\in k^\times$.
Similarly, $g(t_3)= c' t_3$ for some $c'\in k^\times$.
So $g(t_1^2)=cc' t_1^2$. This forces that $g(t_1)=c'' t_1$.
Therefore $g$ preserves the grading of $A$.

Rewrite $g$ as $g(t_1)=c_1 t_1$, $g(t_2)=c_2 t_2$ and $g(t_3)=
c_3 t_3$. Then $c_3=c_1^2 c_2^{-1}$. Hence $\Aut(A)=k^{\times 2}$.
\end{proof}

\begin{proposition}
\label{xxpro4.6}
Let $A$ be $A(6)$ where  $\beta$ is not a root of unity or
the algebra $A(7)$ where $p$ is not a rot of unity. Then
every algebra automorphism of $A$ preserves the grading of
$A$ and $\Aut(A)=(k^\times )^2$.
\end{proposition}

\begin{proof} The proof for $A(7)$ is very similar to the proof
for $A(6)$. So we assume $A=A(6)$.

By \eqref{E1.5.6}, the Nakayama automorphism of $A$ is given by
$$\mu: x\to -\beta x, \quad y\to -\beta^{-1} y.$$
Let $g$ be any algebra automorphism of $A$. It is well known that
$A$ has a $k$-linear basis $\{ x^{n_1} (yx)^{n_2} y^{n_3}\mid n_1\geq 0\}$.
Let $\{x_1,x_2,x_3\}=\{x, yx, y\}$. Then $\lambda_1=-\beta$, $\lambda_2=1$
and $\lambda_3=-\beta^{-1}$. Since $\beta$ is not a root of unity,
we can apply Lemma \ref{xxlem4.3} to this situation. Hence $g(x)=cx$
for some $c\in k^\times$. By symmetry, $g(y)= c' y$ for some $c'\in k^\times$.
The assertion follows.
\end{proof}

Next we deal with non-diagonalizable $\mu_A$, when $A$ is either $A(3)$, or
$A(4)$ or $A(5)$. By the way the automorphism group of the Jordan plane
$$k_{J}[t_1,t_2]:=k\langle t_1,t_2\rangle/((t_2+t_1)t_1=t_1t_2)$$
was given in \cite{Sh}. We will compute the automorphism group of
3-dimensional analogues of $k_{J}[t_1,t_2]$.

We consider a couple of subgroups. If $A$ is ${\mathbb Z}$-graded,
we use $\Aut_{gr}(A)$ for the subgroup of automorphisms that
preserving the ${\mathbb Z}$-grading. If $A$ is connected graded, an
automorphism $g\in \Aut(A)$ is called {\it unipotent} if 
$$g(x)=
x+{\text{higher degree terms}}$$ 
for all homogeneous element $x\in A$.
The subgroup of unipotent automorphisms is denoted by $\Aut_{uni}(A)$.
If $I$ is an ideal of $A$ (of codimension $1$),
let $\Aut(I)$ be the subgroup of $\Aut(A)$ consisting of $g$ preserving
$I$. The following lemma is easy and known and the proof is omitted.

\begin{lemma}
\label{xxlem4.7} Let $A$ be a connected graded algebra.
\begin{enumerate}
\item[(1)]
Let $x$ and $y$ be two nonzero elements in $A$ such that
$xy=qyx$ for some $1\neq q\in k$. Then $x,y\in A_{\geq 1}$.
As a consequence, $g(x), g(y)\in A_{\geq 1}$ for all $g\in \Aut(A)$.
\item[(2)]
Let $A$ be generated in degree one. Then
$\Aut(A_{\geq 1})=\Aut_{gr}(A)\ltimes \Aut_{uni}(A)$.
\end{enumerate}
\end{lemma}

For most common noncommutative connected graded algebras,
$\Aut(A_{\geq 1})=\Aut(A)$. So the above lemma tells
us that we should work on two subgroups $\Aut_{gr}(A)$
and $\Aut_{uni}(A)$. Let $\phi$ be any $k$-linear map of
$A$, an element $a\in A$ is called a $\phi$-eigenvector 
(associated to an eigenvalue $\lambda$) if
$\phi(a)=\lambda a$.

\begin{lemma}
\label{xxlem4.8}
Let $A$ be an AS regular algebra generated by $A_1$ and
$\mu_A$ be the Nakayama automorphism of $A$. Suppose $\mu_A$
has eigenvalues $\{\lambda_1,\cdots,\lambda_n\}$ with respect
to the basis $\{v_1,\cdots,v_n\}$ of $A_1$. Assume that $\lambda_i
\neq \lambda_j$ for all $i\neq j$, then every graded algebra
automorphism of $A$ is of the form
$$g: v_i \to c_i v_i$$
for some $c_i\in k^\times$. As a consequence, $\Aut_{gr}(A)$
is a subgroup of $(k^\times)^n$, which is abelian.
\end{lemma}

\begin{proof} Since $g$ is graded, $g(v_i)\in A_1$. By
\cite[Theorem 3.11]{RRZ1}, $\mu_A g=g \mu_A$. So $g(v_i)$
is a $\mu_A$-eigenvector associated to the eigenvalue
$\lambda_i$. Since $\lambda_i\neq \lambda_j$ for all $i\neq j$,
$g(v_i)=c_i v_i$ for some $c_i\in k$. The assertion follows.
\end{proof}

Using the similar ideas we have the following.

\begin{lemma}
\label{xxlem4.9}
Let $A$ be an AS regular algebra generated by $\{t_1,t_2,t_3\}$.
\begin{enumerate}
\item[(1)]
Suppose $\mu_A$ maps
$$\mu_A: t_1\to a t_1, \quad t_2\to a t_2+b t_1, \quad t_3\to c t_3$$
where $a,b,c\in k^\times$ and $a\neq c$. Then every graded algebra
automorphism $g$ is of the form
\begin{equation}
\label{E4.9.1}\tag{E4.9.1}
g: t_1\to c_1 t_1, \quad t_2\to c_1 t_2+ c_2 t_1, \quad t_3\to c_3 t_3
\end{equation}
where $c_1, c_3\in k^\times$ and $c_2\in k$.
\item[(2)]
Suppose $\mu_A$ maps
$$\mu_A: t_1\to t_1, \quad t_2\to t_2+b_1 t_1, 
\quad t_3\to t_3+ b_1 t_2+ b_2 t_3$$
where $b_1\in k^\times$ and $b_2\in k$. Then every graded algebra
automorphism $g$ is of the form
\begin{equation}
\label{E4.9.2}\tag{E4.9.2}
g: t_1\to a t_1, \quad t_2\to a t_2+ c_1 t_1, 
\quad t_3\to a t_3+c_1 t_2+c_2 t_1
\end{equation}
where $a\in k^\times$ and $c_1, c_2\in k$.
\item[(3)]
Suppose $q^3\neq 1$. The group $\Aut_{gr}(A(3))$ 
consists of all maps of the form
\eqref{E4.9.1}. If $G$ is a finite subgroup of $\Aut_{gr}(A)$, 
then $G$ is a subgroup
of $(k^\times )^2$.
\item[(4)]
Suppose $p^3\neq 1$. The group $\Aut_{gr}(A(4))$ consists 
of all maps of the form
\eqref{E4.9.1}. If $G$ is a finite subgroup of $\Aut_{gr}(A)$, 
then $G$ is a subgroup
of $(k^\times )^2$.
\item[(5)]
The group $\Aut_{gr}(A(5))$ consists of all maps of the form
\eqref{E4.9.2}. If $G$ is a finite subgroup of $\Aut_{gr}(A)$, 
then $G$ is a subgroup
of $k^\times $.
\end{enumerate}
\end{lemma}

\begin{proof} Use linear algebra and
\cite[Theorem 3.11]{RRZ1}, namely, $\mu_A g=g \mu_A$. Details are omitted.
\end{proof}

To prove every algebra automorphism
preserves the grading, we need to show that $\Aut_{uni}(A)$
is trivial. The following lemma is useful in computation.

\begin{lemma}
\label{xxlem4.10}
Let $a, b\in k^\times$ and $c\in k$.
Let $\phi$ be an algebra automorphism of
the Jordan plane $k_{J}[t_1,t_2]$
of the form
$$\phi: t_1\to a t_1, \quad t_2\to a t_2+b t_1.$$
Then any $\phi$-eigenvector in $k_{J}[t_1,t_2]$ is a polynomial of
$t_1$.
\end{lemma}

\begin{proof} Let $f$ be a $\phi$-eigenvector. Since $\phi$ is a graded
algebra automorphism, we may assume that $f$ is homogeneous of degree
$d$. Write $f=\sum_{i=0}^d c_i t_1^{d-i} t_2^i\neq 0$. Let $s$ be the integer
such that $c_{s}\neq 0$ and $c_{i}=0$ for all $i>s$. It remains to show that
$s=0$. If not, we assume $c_s=1$ and then $f=t_1^{d-s} t_2^s+ c_{s-1} t_1^{d-s+1}t_2^{s-1}+\ldt$
where $\ldt$ is (any element) of the form $\sum_{i<s-1} c_i t_1^{d-i}t_2^i$.
Since $f$ is a $\phi$-eigenvector, we have
$$\begin{aligned}
\lambda f&=\phi(f)=(at_1)^{d-s} (at_2+bt_1)^s+ c_{s-1} (at_1)^{d-s+1}(at_2+bt_1)^{s-1}+\ldt\\
&=a^d t_1^{d-s} t_2^s+s a^{d-1}b t_1^{d-s+1} t_2^{s-1}+a^d c_{s-1} t_1^{d-s+1}t_2^{s-1}+\ldt.
\end{aligned}
$$
Hence $\lambda =a^d$, $s a^{d-1}b=0$, which yield a contradiction.
Therefore $s=0$ and the assertion follows.
\end{proof}

\begin{proposition}
\label{xxpro4.11}
Let $A$ be an AS regular domain generated by $\{t_1,t_2,t_3\}$. Assume
that $\{t_1^{n_1}t_2^{n_2}t_3^{n_3}\mid n_i\geq 0\}$ is a
$k$-linear basis of $A$ and that the subalgebra generated by $t_1,t_2$
is the Jordan plane with relation $(t_2+t_1)t_1=t_1t_2$ and that
$t_3 t_1=at_1 t_3$ for some $a\in k^\times$.
\begin{enumerate}
\item[(1)]
Suppose $\mu_A$ maps
$$\mu_A: t_1\to a t_1, \quad t_2\to a t_2+b t_1, \quad t_3\to c t_3$$
where $a$ is not a positive power of $c$ and $c$ is not a positive
power of $a$ and
both $a$ and $c$ are not roots of unity.
Then every unipotent algebra automorphism is the identity on $t_1$ and $t_3$.
\item[(2)]
Suppose $q$ is not a root of unity. Then $\Aut(A(3))=\Aut_{gr}(A(3))$
which consists of all maps of the form \eqref{E4.9.1}.
If $G$ is a finite subgroup
of $\Aut(A)$, then $G$ is a subgroup of $(k^\times )^2$.
\item[(3)]
Suppose $q$ is not a root of unity. The center of $A(3)$ is $k$.
\item[(4)]
Suppose $p$ is not a root of unity. Then  $\Aut(A(4))=\Aut_{gr}(A(4))$
which consists of all maps of the form \eqref{E4.9.1}.
If $G$ is a finite subgroup
of $\Aut(A)$, then $G$ is a subgroup of $(k^\times )^2$.
\item[(5)]
Suppose $p$ is not a root of unity. The center of $A(4)$ is $k$.
\end{enumerate}
\end{proposition}

\begin{proof} (1) Since $t_3$ is a $\mu_A$-eigenvector associated
to $c$, so is $v: = g(t_3)$ by Theorem \ref{xxthm4.2}.
Write $v=\sum_{n\geq 0} f_n(t_1,t_2) t_3^n$. Then
$\mu_A(v)=\sum_{n\geq 0} \mu_A(f_n)c^n t_3^n$. So each $f_i$
is a $\mu_A$-eigenvector. By Lemma \ref{xxlem4.10},
$f_0= w_0 t_1^{n_0}$. So $f_0$ is a $\mu_A$-eigenvector associated
to $a^{n_0}$. So $c=a^{n_0}$. By the hypothesis, $c$ is not a
power of $a$. So $f_0=0$, consequently, $v=h t_3$ for some $h\in A$.
Similarly, $g^{-1}(t_3)=h' t_3$. Then $h$ and $h'$ are units and
whence $g(t_3)=t_3$.

Since $t_1$ is a $\mu_A$-eigenvector associated
to $a$, so is $w: = g(t_1)$ by Theorem \ref{xxthm4.2}.
Let $w=g(t_1)$ and write $w=\sum_{n\geq 0} f_n(t_1,t_2) t_3^n$
by recycling the notation from the last paragraph. For
each $n$, $f_n t_3^n$ is a $\mu_A$-eigenvector.
In particular, $f_n$ is a $\mu_A$-eigenvector. By
Lemma \ref{xxlem4.10},  $f_n =w_n t_1^{d_n}$.
So $w$ is generated by $t_1$ and $t_3$.
Now we can write $w=\sum t_1^{n} h_n(t_3)$.
Then $h_0(t_3)$ is a $\mu_A$-eigenvector associated
to $a$. If $h_0\neq 0$,  this is impossible as $a$ is not a power of
$c$. So $h_0=0$ and $g(t_1)=w=t_1 h$. Similarly, $g^{-1}(t_1)
t_1 h'$. Then $h$ and $h'$ are units and
whence $g(t_1)=t_1$.

Now let $u=g(t_2):= t_2+h$. Since $\mu_A(t_2) =at_2+b t_1$,
one sees that $h$ is a $\mu_A$-eigenvector. So by Lemma
\ref{xxlem4.10}, $h$ is generated by $t_1$ and $t_3$.
Applying $g$ to the relation $t_2t_1=t_1t_2-t_1^2$, one sees that
$ht_1=t_1h$. By the relation $t_3 t_1=at_1t_3$, we have
$h\in k[t_1]_{\geq 2}$ (since $g$ is unipotent).

In the following proof, $A$ is either $A(3)$ or $A(4)$ with
parameter not a root of unity.

(2,4) One can easily show that $\Aut(A)=\Aut(A_{\geq 1})$.
By Lemma \ref{xxlem4.7}, it suffices to show that
$g\in \Aut_{uni}(A)$ is the identity.
By part (1), $g(t_1)=t_1, g(t_3)=t_3$ and $g(t_2)=t_2+h$
where $h\in k[t_1]_{\geq 2}$. Then the third relation implies
that $h=0$. The assertion follows.

(3,5) Let $f$ be the center of $A$. Write $f=\sum_{n\geq 0}
f_n t_3^n$ where $f_n$ is in the subalgebra generated by
$t_1$ and $t_3$. Since $f$ is a $\mu_A$-eigenvector, by the
form of $\mu_A$, each $f_n$ is a  $\mu_A$-eigenvector.
By Lemma \ref{xxlem4.10}, $f_n\in k[t_1]$. So $f$
is in the subalgebra generated by $t_1$ and $t_3$.
Since we have $t_3t_1=q t_1t_3$ in $A(3)$ and
$t_3t_1=pt_1t_3$ in $A(4)$, $f$ commutes with both
$t_1$ and $t_3$ if and only if $f\in k$. Therefore
$f\in k$ and the center of $A$ is trivial.
\end{proof}

\section{Automorphisms of $A(5)$}
In this section let $A$ be the algebra $A(5)$ defined in
the introduction. The relations \eqref{E0.3.5} are equivalent
to the following relations
\begin{equation}
\label{E5.0.1}\tag{E5.0.1}
t_i (t_j-t_{j-1})=t_j (t_i-t_{i-1})
\end{equation}
for all $1\leq i<j \leq 3$, where $t_0=0$ by convention.
We also set $\deg t_i=1$ for all $i=1,2,3$.

Given a ${\mathbb Z}$-graded  algebra $C$ and  a graded
algebra automorphism $\tau$ of $C$, the (right) graded twist of
$C$ associated to $\tau$, denoted by $C^{\tau}$, is defined as
follows: as a graded $k$-vector space, $C^{\tau}=C$, the
multiplication $\ast$ of $C^{\tau}$ is given by
\begin{equation}
\label{E5.0.2}\tag{E5.0.2}
 f \ast g= f \tau^{\deg f}(g)\in C=C^\tau
\end{equation}
for all homogeneous elements  $f,g\in C=C^\tau$. We use slightly
different notation from Section \ref{xxsec1}. From now on let
$C$ be the polynomial ring $k[t_1,t_2,t_3]$.

\begin{lemma}
\label{xxlem5.1} The algebra $A$ is {\rm{(}}isomorphic to{\rm{)}}
the graded twist $C^{\sigma}$ where $\sigma$ is the graded
algebra automorphism of $C$ sending $t_i$ to $\sum_{j=1}^i t_j$
for all $i=1,2,3$.
\end{lemma}

\begin{proof} By definition \cite{Zh} or \eqref{E5.0.2}, the
graded twist $C^{\sigma}$ of the commutative ring $C$, with new
multiplication $\ast$, is generated by $t_1,t_2,t_3$ and subject
to the relations
\begin{equation}
\label{E5.1.1}\tag{E5.1.1}
t_i \ast \sigma^{-1}(t_j)=t_i t_j = t_j t_i=t_j \ast \sigma^{-1}(t_i)
\end{equation}
for all $i<j$. Since $\sigma^{-1}(t_i)=
t_i-t_{i-1}$ for all $i=2,3$ and $\sigma(t_1)=t_1$.
So \eqref{E5.1.1} agrees with \eqref{E5.0.1}.
Therefore $A = C^{\sigma}$.
\end{proof}

Let $\phi$ be a $k$-linear endomorphism of some vector space,
say $W$. An element $f\in W$ is said to be a $\phi$-eigenvector
if $\phi(f)=c f$ for some $c\in k$, and $W$ is a $\phi$-invariant
if $\phi(f)=f$, namely $f$ is a $\phi$-eigenvector associated to
eigenvalue $1$.

\begin{lemma}
\label{xxlem5.2} Let $A$ and $C$ be defined as above.
\begin{enumerate}
\item[(1)]
Let $f\in C$. Then $f$ is a $\sigma$-eigenvector if and only if it
is $\sigma$-invariant.
\item[(2)]
$t_1\in C$ is $\sigma$-invariant.
\item[(3)]
$y_2:=t_2^2+ t_1 t_2-2 t_1 t_3\in C$ is a $\sigma$-invariant.
\item[(4)]
If $f\in k[t_1,t_2]$ is $\sigma$-invariant, then $f\in k[t_1]$.
As a consequence, $k[t_1,t_2]^{\langle \sigma\rangle} =k[t_1]$.
\item[(5)]
Every $\sigma$-invariant is generated by
$t_1$ and $y_2$. As a consequence,
$C^{\langle \sigma\rangle} =k[t_1,y_2]$.
\end{enumerate}
\end{lemma}

\begin{proof}
(1) Let $V=\oplus_{i=1}^3 kt_i$. Then $\sigma$ is
unipotent on $V$. Since $C$ is generated by $V$,
$\sigma$ is unipotent on the homogeneous part
of $C$ of degree $d$ for any $d$. Thus, the only
eigenvalue of $\sigma$ on $C$ is 1. So
any $\sigma$-eigenvector is $\sigma$-invariant.

(2) Clear.

(3) By direct computation.

(4) Since $\sigma$ preserves the degree, we may assume
that $f$ is homogeneous. Let $f=\sum_{i=0}^s a_i t_1^{d-i} t_2^i$
where $a_i\in k$ for all $i=0,\cdots, s$ and $a_s\neq 0$ for some
$s\leq d$. We claim
that $s=0$. Since $f$ is $\sigma$-invariant,
$$\begin{aligned}
0&=\sigma(f)-f=\sum_{i=0}^s a_i t_1^{d-i} ((t_1+t_2)^{i}-t_2^{i})\\
&=a_s t_1^{d-s} (s t_1 t_2^{s-1}+\ldt)+\ldt\\
\end{aligned}
$$
where $\ldt$ means some polynomials of $t_2$-degree
less than $s-1$. So we have $a_s s=0$. Since $a_s
\neq 0$, $s=0$ as desired. So $f\in k[t_1]$. The consequence
is clear.

(5) First of all, $t_1$ and $y_2$ are algebraically independent and
elements in $k[t_1,y_2]$ are $\sigma$-invariants by parts (2,3).

Now let $f$ be a $\sigma$-invariant. Again we may assume
that $f$ is homogeneous. Let $g=t_1^d f$ where $d$ is
the degree of $f$. Then $g$ is a polynomial in $t_1, t_2$
and $t_1t_3$.  Write $t_1t_3$ as $\frac{1}{2} (t_2^2+t_1t_2-y_2)$.
Then $g$ is a polynomial in $t_1,t_2,y_2$.
Write $g=\sum_{i=0}^d g_i(t_1,t_2)y_2^i$ where
$g_i(t_1,t_2)$ is a polynomial in $t_1$ and $t_2$ for
each $i$. Since $g,y_2$ are $\sigma$-invariant, by using the
fact that $t_1,t_2,y_2$ are algebraically independent,
each $g_i(t_1,t_2)$ is $\sigma$-invariant. By part (4), $g_i(t_1,t_2)
=c'_i t_1^{d_i}$ where $c'_i\in k$. Since the degree of $g$ is $2d$,
we may write  $g=\sum_{i=0}^s c_i t_1^{2d-2i} y_{2}^i$
where $c_i\in k$ and $c_s\neq 0$. Thus $f=\sum_{i=0}^s c_i t_1^{d-2i} y_{2}^i$
and we need to show that $d-2i\geq 0$ for any $i$ with $c_i\neq 0$.
Considering $g$ as a polynomial in $t_2$ and using the fact $y_2=
t_2^2+t_1t_2-2t_1t_3$, the coefficient of the leading term
of $f$ is $c_s t_1^{d-2s}$. Thus $d-2s\geq 0$ and $f$ is a
polynomial in $t_1$ and $y_2$.
\end{proof}

Let $\tau$ be an algebra automorphism of an algebra
$B$ and let $f$ be a nonzero element in $B$. We say $f$ is $\tau$-normal
if $fx= \tau(x)f $ for all $x\in B$.
By the relations \eqref{E5.0.1} or \eqref{E0.3.5},
$t_1$ in $A$ is a $\sigma$-normal element.

\begin{lemma}
\label{xxlem5.3}
Let $B$ be a ${\mathbb Z}$-graded domain. If $f\in B$ is a $\tau$-normal
element for some automorphism $\tau$, then $\tau$ is a graded algebra
automorphism of $B$. As a consequence, every nonzero homogeneous component
of $f$ is $\tau$-normal.
\end{lemma}

\begin{proof}
For each element $x\in B$ we can define lower degree $l.\deg(x)$ and
upper degree $u.\deg(x)$. Then $x$ is a homogeneous element if and only if
$l.\deg(x)=u.\deg(x)=\deg(x)$. Since $f$ is $\tau$-normal,
$ f x=\tau(x) f$. Then
$$u.\deg(f)+u.\deg(x)=u.\deg (fx)=u.\deg(\tau(x)f)=u.\deg(\tau(x))
+u.\deg(f).$$
Hence $u.\deg(x)=u.\deg(\tau(x))$. Similarly, $l.\deg(x)=l.\deg(\tau(x))$.
This implies that $x$ is homogeneous if and only if $\tau(x)$ is homogeneous
of the same degree. So the assertion follows.
\end{proof}

If $G$ is a subgroup of $\Aut(B)$, the fixed subring of $B$ under
$G$-action is denoted by $B^G$. The following lemma is well known.

\begin{lemma}
\label{xxlem5.4} Let $B$ be a graded algebra and $\tau$ a graded algebra
automorphism. Let $G$ be a subgroup of $\Aut_{gr}(B)$ such that $g \tau=\tau g$
for all $g\in G$. Then $G$ is naturally a subgroup of $\Aut_{gr}(B^\tau)$ and
the fixed subring $(B^\tau)^G$ is a graded twist of $(B^G)^\tau$.
\end{lemma}

\begin{proof} As a graded vector space $B^\tau=B$. Let $\cdot$ (respectively
$\ast$) be the multiplication of $B$ (respectively, $B^\tau$). Let $g$ be
a $k$-linear graded automorphism of $B=B^\tau$. Since
$x \ast y=x \cdot \tau^{\deg x}(y)$ and since $g\in G$ commutes with
$\tau$, $g$ is an algebra endomorphism of $B$ if and only if it
is an algebra endomorphism of $B^\tau$.
Therefore $G$ is a subgroup of $\Aut_{gr}(B^\tau)$.
As a graded vector space, it is clear that $(B^\tau)^G=B^G$.
Since $\tau$ commutes with $G$, $\tau$ is naturally
a graded automorphism of $B^G$ by restriction.
By comparing multiplications of $(B^\tau)^G$ and $B^G$,
one sees that $(B^\tau)^G=(B^G)^\tau$ as a graded algebra.
\end{proof}

Considering $y_2=t_2^2+t_1t_2-2t_1t_3\in C$ as an element
in $A$, we have
$$
y_2=t_2\ast t_2 +2t_1\ast t_2-2t_1 \ast t_3
$$
where $\ast$ is the multiplication in $A=C^\sigma$. When $\ast$
is omitted,
\begin{equation}
\label{E5.4.1}\tag{E5.4.1}
y_2=t_2^2+2t_1t_2-2t_1t_3\in A.
\end{equation}

\begin{lemma}
\label{xxlem5.5}
Let $\sigma$ be the automorphism of $C$ and $A$ determined by
sending $t_i$ to $\sum_{j=1}^i t_j$ for all $i=1,2,3$.
\begin{enumerate}
\item[(1)]
The Nakayama automorphism of $A$ is $\sigma^{-3}$.
\item[(2)]
Fix any integer $d$, then $f\in A$ is $\sigma^d$-invariant, or a
$\sigma^d$-eigenvector, if and only if $f\in C$ is a $\sigma^d$-invariant.
\item[(3)]
If $\tau$ is a graded algebra automorphism of $A$, then $\tau(t_1)=c t_1$
for some $c\in k^\times$.
\item[(4)]
If $f\in A$ is a nonzero $\tau$-normal element for some automorphism $\tau$,
then $\tau=\sigma^d$ and $f$ is homogeneous of degree $d$.
\item[(5)]
An element $f\in A$ is a normal element if and only if $f$ is homogeneous
and $\sigma$-invariant. As a consequence,
normal elements in $A$ are precisely these homogeneous
elements in $A^{\langle \sigma \rangle}=C^{\langle \sigma \rangle}$.
\item[(6)]
The subspace 
$kt_1$ consists of all possible elements $f$ with the properties:
$f$ is $\tau$-normal and $\tau^3=\mu_{A}^{-1}$. As a consequence,
$kt_1$ is preserved by any algebra automorphism of $A$.
\item[(7)]
Any algebra automorphism of $A$ commutes with $\sigma$.
\item[(8)]
Any algebra automorphism of $A$ becomes a graded algebra automorphism
when restricted to $k[t_1,y_2]$.
\end{enumerate}
\end{lemma}

\begin{proof}
(1) By direct computation, we have $\sigma^3(t_1)=t_1, 
\sigma^3(t_2)=3t_1+t_2$ and $\sigma^3(t_3)=6t_1+3t_2+t_3$
since $\sigma(t_i)=\sum_{j=1}^i t_j$ for all $i=1,2,3$.  
Hence  $\sigma^{3}=\mu_A^{-1}$ by  \eqref{E1.5.5}.

(2)  Let $G$ be the subgroup of $\Aut_{gr}(C)$ generated
by $\sigma^d$. By the proof of Lemma \ref{xxlem5.4},
$C^G=A^G$. The assertion follows.

(3) This follows from the fact that $t_1$ is the only normal
element in degree 1.

(4) By Lemma \ref{xxlem5.3}, $\tau$ is a graded algebra
automorphism and there is a nonzero homogeneous $\tau$-normal
component of $f$, say $f_d$, of degree $d$. Then
$$\tau(t_i) \ast f_d=f_d \ast t_i=f_d \sigma^d(t_i)
=\sigma^d(t_i) f_d=\sigma^d(t_i)\ast
\sigma^{-1} (f_d).$$
Let $i=1$ in above, we have
$$c t_1\ast f_d=\tau(t_1)\ast f_d=t_1\ast \sigma^{-1}(f_d)$$
for some $c\in k^\times$. So $f_d$ is a $\sigma$-eigenvector, so a
$\sigma$-invariant. Using the fact $f_d$ is a $\sigma$-invariant,
the equation now becomes
$$\tau(t_i) \ast f_d= \sigma^d(t_i) \ast f_d$$
or $\tau(t_i)=\sigma^d(t_i)$ for all $i$.

(5) One implication is part (4). For the other implication,
let $f$ be a homogeneous $\sigma$-invariant. Then, for all
$y\in A=C$,
$$ f\ast y= f \sigma^{\deg f} (y)=\sigma^{\deg f}(y) f=
\sigma^{\deg f}(y) \ast f$$
which implies that $f$ is normal.

(6) If $f$ is $\tau$-normal, by part (4), $\tau=\sigma^d$ for some $d$
and $f$ is homogeneous of degree $d$. If $\tau^3=\mu_{A}^{-1}=\sigma^3$,
then $d=1$ (as $\sigma$ has infinite order). Hence $f$ has degree 1,
and one can easily check that the only degree 1 normal element is $kt_1$.

Let $\phi$ be any algebra automorphism of $A$ and let $f=\phi(t_1)$.
Then $f$ is $\tau$-normal with $\tau=\phi \sigma \phi^{-1}$. Since
$\phi$ commutes with $\mu_{A}^{-1}=\sigma^3$ [Theorem 0.6], we have
$\tau^3=\sigma^{3}=\mu_{A}^{-1}$. By the above paragraph, $f\in kt_1$.
Therefore $\phi$ preserves $kt_1$.

(7) Let $\phi$ be an algebra automorphism of $A$. By part (6),
$\phi(t_1)=ct_1$ for some $c\in k^\times$. Applying $\phi$ to
the equation $t_1 x= \sigma(x) t_1$, one obtains that $\phi$
commutes with $\sigma$.

(8) By part (6) any automorphism $\phi$ sends $t_1$ to $ct_1$.
Now $y_2$ is a $\sigma^2$-normal element. Since $\phi$
commutes with $\sigma$, $\phi(y_2)$ is a $\sigma^2$-normal element.
By part (4), $\phi(y_2)$ is homogeneous of degree 2.
Thus $\phi(y_2)=\alpha y_2+ \beta t_1^2$ for some
$\alpha,\beta\in k$. The assertion follows.
\end{proof}

\begin{lemma}
\label{xxlem5.6}
Let $g$ be an algebra automorphism of $A$ such that $g(t_1)=t_1$.
\begin{enumerate}
\item[(1)]
$g(t_2)$ and $g(t_3)$ have zero constant terms.
\item[(2)]
If $g$  is  a graded algebra automorphism, then
$g(t_2)=t_2+a t_1$ and $g(t_3)=t_3+a t_2+d t_1$
for $a,d\in k$.
\item[(3)]
Let $g(t_2)=t_2+v$. Then
$v\in k[t_1,y_2]$.
\end{enumerate}
\end{lemma}

\begin{proof}
(1) Applying $g$ to the relation $t_2 (t_3-t_2)=t_3 (t_2-t_1)$,
one sees that $g(t_2)$ has zero constant term. Applying
$g$ to the relation $t_2 (t_3-t_2)=t_3 (t_2-t_1)$ again,
one sees that $g(t_3)$ has zero constant term.

(2) Let $g(t_2)=a t_1+ bt_2+ ct_3$ and $g(t_3)=d t_1+ e t_2+ft_3$.
Applying $g$ to the relations, we see that $c=0$, $b=1$, $f=1$
and $e=a$. So the assertion follows.

(3) By part (2), $g(t_2)=t_2 +a t_1+ f$ where $f$ has lower degree
at least 2. Applying $g$ to the relation $t_1(t_2-t_1)=t_2t_1$,
one sees that $t_1f=ft_1$. So $f$ is a $\sigma$-invariant.
By Lemmas \ref{xxlem5.2}(5) and \ref{xxlem5.4}, $f$ is
generated by $t_1$ and $y_2$. So $v=at_1+f\in k[t_1,y_2]$.
\end{proof}

Now let $\partial: k[t_1,y_2]\to k[t_1,y_2]$ be the
derivation sending $f$ to $(\deg f) f$ for any homogeneous
element $f\in k[t_1,y_2]$.
We will prove that, given any  $u\in k[t_1,y_2]$ and any $\lambda
\in k$, the following determines an algebra automorphism
of $A$:
\begin{equation}
\label{E5.6.1}\tag{E5.6.1}
g(u,\lambda): \qquad
\begin{aligned}
t_1 & \to t_1,\\
t_2 & \to t_2+ t_1 u, \\
t_3 & \to t_3 +w, \quad {\text{where}}\; w=ut_2+ \frac{1}{2}
[ut_1-\partial (u)t_1+ u^2t_1-\lambda t_1].
\end{aligned}
\end{equation}

\begin{lemma}
\label{xxlem5.7}
Let $v$ be an element in $k[t_1,y_2]$.
\begin{enumerate}
\item[(1)]
$t_1 v= vt_1$.
\item[(2)]
$t_2 v= vt_2- \partial (v) t_1$.
\item[(3)]
$t_3 v=v t_3 - \partial(v) t_2 + \frac{1}{2}(\partial^2 -\partial)(v) t_1$.
\item[(4)]
Suppose $g$ is an algebra endomorphism of $A$ such that
$g(t_1)=t_1$ and $g(t_2)=t_2+v$ where $v\in k[t_1,y_2]$
and $g(t_3)=t_3+w$ and that $g(y_2)=y_2+\lambda t_1^2$. Then
$v=t_1 u$ where $u\in k[t_1,y_2]$ and $w= ut_2+\frac{1}{2}
[ut_1-\partial (u)t_1+ u^2t_1-\lambda t_1]$.
Namely, $g$ is of the form \eqref{E5.6.1}.
\item[(5)]
Let $g$, $u$ and $w$ be as in \eqref{E5.6.1} and $v=ut_1$. Then
$\sigma^{-1}(w)= w-v$. Assume $g$ is an algebra endomorphism, then
$g$ commutes with $\sigma$.
\item[(6)]
Let $g=g(u,\lambda)$ be as in \eqref{E5.6.1}. Then it is an 
algebra automorphism with inverse
$h: =g^{-1}$ given by
$$\begin{aligned}
h(t_1)& =t_1\\
h(t_2)&=t_2 + u' t_1\\
h(t_3)&=t_3+ w', \qquad {\text{where}} \\
u'&=-u(t_1,y_2-\lambda t_1^2)\\
w'&= u't_2+ \frac{1}{2}
[u't_1-\partial (u')t_1+ (u')^2t_1-(-\lambda) t_1].
\end{aligned}
$$
Using the notation introduced in \eqref{E5.6.1}, $h$
corresponds to $g(u',-\lambda)$  in \eqref{E5.6.1}
determined by the parameters $(u', -\lambda)$.
\end{enumerate}
\end{lemma}

\begin{proof} (1,2,3) By direct computation.

(4) We have $$\begin{aligned}
y_2+\lambda t_1^2&=g(y_2)=(t_2+v)^2+2t_1(t_2+v)-2t_1 (t_3+w)\\
&=t_2^2+ vt_2+t_2 v+v^2 +2t_1t_2+2t_1v-2t_1t_3 -2 t_1 w\\
&= y_2+2v t_2-\partial(v) t_1+v^2 +2t_1v -2t_1 w.\\
\end{aligned}
$$
So $\lambda t_1^2=2v t_2  -\partial(v) t_1+v^2 +2t_1v -2t_1 w$.
Thus $t_1$ divides $v$. Write $v=t_1 u$. Then 
$\lambda t_1^2 = 2ut_1t_2-\partial(u)t_1^2+ u^2t_1^2+ut_1^2-2t_1 w$. 
Hence $w= ut_2+\frac{1}{2}
[ut_1-\partial (u)t_1+ u^2t_1-\lambda t_1]$.

(5) Direct computation.

(6) First we show that $g$ determines an algebra endomorphism
of $A$. Since $ut_1$ commutes with $t_1$, $g$ preserves the relation
$t_1 (t_2-t_1)=t_2 t_1$. It is easy to see that $g$
preserves the relation $t_1(t_3-t_2)=t_3 t_1$ if and
only if $\sigma^{-1}(w)=w-ut_1$. So, by part (5),
$g$ preserves this relation. Next we show that $g$ preserves
the relation $t_3(t_2-t_1)=t_2(t_3-t_2)$. Let $L=t_3(t_2-t_1)$
and $R=t_2(t_3-t_2)$. Write $w=ut_2+s$ where
$s=\frac{1}{2}
[ut_1-\partial (u)t_1+ u^2t_1-\lambda t_1]$.
Then
$$\begin{aligned}
g(L)&= (t_3+w)(t_2+ut_1 -t_1)=L+wt_2+w(ut_1-t_1)+t_3ut_1\\
&=L+ut_2^2 +s t_2 +ut_2 (ut_1-t_1)+s(ut_1-t_1)\\
&\qquad\qquad\qquad
+(ut_1) t_3-\partial(ut_1)t_2 +\frac{1}{2}(\partial^2-\partial)(ut_1) t_1\\
&=L+ut_2^2 +s t_2 +u(ut_1-t_1)t_2-u\partial(ut_1-t_1)t_1
+s(ut_1-t_1)\\
&\qquad\qquad\qquad
+(ut_1) t_3-\partial(ut_1)t_2 +\frac{1}{2}(\partial^2-\partial)(ut_1) t_1\\
\end{aligned}
$$
and
$$\begin{aligned}
g(R)&=(t_2+ut_1)(t_3-t_2+w-ut_1)\\
&=R+ t_2(w-ut_1)+ut_1(t_3-t_2)
+ut_1(w-ut_1)\\
&=R+ut_2^2 -\partial(u)t_1 t_2+(s-ut_1)t_2-\partial(s-ut_1)t_1\\
&\qquad\qquad +ut_1(t_3-t_2)
+ut_1(w-ut_1).\\
\end{aligned}
$$
So
$$\begin{aligned}
g(L)-g(R)&=u^2t_1t_2-u\partial(ut_1-t_1)t_1+s(ut_1-t_1) 
+\frac{1}{2}(\partial^2-\partial)(ut_1)t_1\\
&\quad\qquad
+\partial(s-ut_1)t_1-ut_1(ut_2+s-ut_1) \\
&=-u[\partial(u)t_1+ut_1-t_1]t_1 +s(ut_1-t_1) 
+\frac{1}{2}[\partial^2(u)t_1+\partial(u)t_1]t_1 \\
&\quad\qquad +\partial(s)t_1-\partial(u)t_1^2-ut_1^2-ut_1s+u^2t_1^2\\
&=-u\partial(u)t_1^2 + s(ut_1-t_1)
+\frac{1}{2}\partial^2(u)t_1^2-\frac{1}{2}\partial(u)t_1^2
+\partial(s)t_1-ut_1s\\
&=-u\partial(u)t_1^2 +\frac{1}{2}
[ut_1-\partial (u)t_1+ u^2t_1-\lambda t_1](ut_1-t_1)
+\frac{1}{2}\partial^2(u)t_1^2\\
&\quad\qquad  + \frac{1}{2}[ut_1-\partial^2(u)t_1
+(\partial(u)u+u\partial(u))t_1+u^2t_1-\lambda t_1]t_1\\
&\quad\qquad -\frac{1}{2}\partial(u)t_1^2 
-\frac{1}{2}ut_1[ut_1-\partial (u)t_1+ u^2t_1-\lambda t_1] \\
&=0
\end{aligned}
$$
by using the equations $s=\frac{1}{2}
[ut_1-\partial (u)t_1+ u^2t_1-\lambda t_1]$ and $u\partial(u)=\partial(u)u$.
 Therefore $g$ is an algebra endomorphism.

One can also check that $g(y_2)=y_2+\lambda t_1^2$. The map $h$ is of the form
$g$ with $(u,\lambda)$ being replaced by $(u',-\lambda)$. So $h$ is an algebra
endomorphism. Similar to the $g$, we have $h(y_2)=y_2-\lambda t_1^2$.
By the choice of $u'$, one sees that $h(g(t_2))=t_2$. So we have
$h(g(t_1))=t_1, h(g(t_2))=t_2$ and $h(g(y_2))=y_2$. Since $y_2=t_1^2-2t_1t_3$,
we also have $h(g(t_3))=t_3$. Thus $h$ is a left inverse of $g$. 
A similar proof
shows that $g$ is a left inverse of $h$. Therefore $h$ is the inverse of $g$.
\end{proof}

\begin{theorem}
\label{xxthm5.8}
Every algebra automorphism $\tau$ of $A=A(5)$ is determined by
$$\tau(a, u, \lambda): \qquad
\begin{aligned}
t_1 & \to a t_1,\\
t_2 & \to a(t_2+ t_1 u), \\
t_3 & \to a(t_3 +w), \quad {\text{where}}\; w=ut_2+\frac{1}{2}
[ut_1-\partial (u)t_1+ u^2t_1-\lambda t_1].
\end{aligned}
$$
for some $a\in k^\times$, $u\in k[t_1,y_2]$ and $\lambda\in k$.
Conversely, given any $a\in k^\times$, $u\in k[t_1,y_2]$ and $\lambda\in k$,
$\tau(a,u,\lambda)$ is an algebra automorphism of $A$ and
the product of two such is determined by
$$\tau(a,u,\lambda) \tau( a', u', \lambda ')=
\tau(aa', a' u(t_1,y_2)+u'(at_1, a^2 (y_2+\lambda t_1^2)), \lambda+\lambda').$$
\end{theorem}

\begin{proof} Let $\tau$ be an algebra automorphism of $A$. By
Lemma \ref{xxlem5.5}(6), $\tau(t_1)=a t_1$ for some
$a\in k^\times$. Replacing $\tau$ by $\tau_0=\xi_{a^{-1}}\circ \tau$,
where $\xi_{a^{-1}}: f\to a^{-\deg f} f$ for any homogeneous element
$f\in A$, we obtain $\tau_0(t_1)=t_1$ (or we have $a=1$ for $\tau_0$).
By Lemma \ref{xxlem5.6}(1),
the constant terms of $\tau_0(t_2)$ and $\tau_0(t_3)$ are zero.
Hence $\tau_0$ preserves the maximal graded ideal ${\mathfrak{m}}
:=A_{\geq 1}$. Since $A=\gr A$ with respect to the ${\mathfrak{m}}$-filtration,
$\gr\tau_0$ is a graded algebra automorphism of $A$. By Lemma
\ref{xxlem5.6}(2), $\gr \tau_0$ sends $t_2\to t_2+ at_1$ and $t_3\to
t_3+a t_2+d t_1$ for some $a,d\in k$. This implies
that $\tau_0$ sends $t_2\to t_2+ a t_1+hdt$ and
$t_3\to t_3 +a t_2+ d t_1+hdt$.
By Lemma \ref{xxlem5.6}(3), $\tau_0(t_2)=t_2+v$ where
$v\in k[t_1,y_2]$. By Lemma \ref{xxlem5.5}(8),
$\tau_0$ maps $y_2$ to $by_2 +\lambda t_1^2$.
Consider the $t_2^2$ term of $y_2$, one sees that
$\tau_0(y_2)=y_2+\lambda t_1^2$. By Lemma
\ref{xxlem5.7}(4), $v=u t_1$ and
$\tau_0(t_3)=t_3+w$ where $w=ut_2+\frac{1}{2}
[ut_1-\partial (u)t_1+ u^2t_1-\lambda t_1]$, namely,
$\tau_0$ is of the form $g(u,\lambda)$ given
in \eqref{E5.6.1}, which is an algebra automorphism
by Lemma \ref{xxlem5.7}(6). Therefore we have $\tau=\xi_{a}\circ
g(u,\lambda)$, which shows the assertion.

Conversely, since $g(u,\lambda)$ is an algebra automorphism,
so is $\tau=\xi_{a}\circ g(u,\lambda)$.

The final statement
about the product can be checked directly.
\end{proof}

\section{Nakayama automorphisms and locally nilpotent derivations}
\label{xxsec5}
In this section we discuss some relationships between the
Nakayama automorphism and locally nilpotent derivations.
Most of the ideas come from \cite{BeZ}.

Recall that a $k$-linear map
$\delta: A\to A$ is called a  derivation if
$$\delta(ab)=a\delta(b)+\delta(a) b$$
for all $a,b \in A$. A derivation $\delta$ is called locally
nilpotent if, for every $a\in A$, there is an $n$
such that $\delta^n(a)=0$. The set of
derivations (respectively, locally nilpotent derivations)
is denoted by ${\rm Der}(A)$ (respectively ${\rm LND}(A)$).
We will deal with a slightly more complicated form.

\begin{definition}
\label{xxdef6.1} Let $A$ be an algebra.
\begin{enumerate}
\item[(1)]
A {\it higher derivation} (or {\it Hasse-Schmidt derivation})
\cite{HS} on $A$ is a sequence of $k$-linear endomorphisms
$\partial:=(\partial_i)_{i\geq 0}$ such that:
$$\partial_0 = id_A, \quad
\text{and} \quad  \partial_n(ab) =\sum_{i=0}^n
\partial_i(a)\partial_{n-i}(b)
$$
for all $a, b \in A$ and for all $n\geq 0$. The set of
higher derivations is denoted by ${\rm Der}^H(A)$.
\item[(2)]
A higher derivation is called {\it iterative} if $\partial_i \partial_j =
{i+j \choose i} \partial_{i+j}$ for all $i, j
\geq 0$.
\item[(3)]
A higher derivation is called {\it locally nilpotent} if
\begin{enumerate}
\item
for all $a \in A$  there exists $n \geq  0$ such that $\partial_i(a) = 0$
for all $i \geq  n$,
\item
The map $G_{\partial,t}: A[t]\to A[t]$ defined by
\begin{equation}
\label{E6.1.1}\tag{E6.1.1}
G_{\partial,t}: a\mapsto \sum_{i\geq 0} \partial_i(a) t^i,\quad 
t\mapsto t
\end{equation}
is an algebra automorphism of $A[t]$.
\end{enumerate}
The set of locally nilpotent higher derivations is denoted by
$\LND^H(A)$.
\item[(4)]
For any $\partial\in {\rm Der}^H(A)$, then the kernel of $\partial$
is defined to be
$$\ker \partial =\bigcap_{i\geq 1} \ker \partial_i.$$
\end{enumerate}
\end{definition}

Note that $\partial_1$ is necessarily a derivation of $A$. Hence
there is a $k$-linear map ${\rm Der}^H(A)\to {\rm Der}(A)$ by sending
$(\partial_i)$ to $\partial_1$. In characteristic $0$, the only
iterative higher derivation $\partial=(\partial_i)$
on A such that $\partial_1=\delta$ is given by:
\begin{equation}
\label{E6.1.2}\tag{E6.1.2}
\partial_n=\frac{\delta^n}{n!}
\end{equation}
for all $n \geq  0$. It is clear that if $\delta$ is
a locally nilpotent derivation, then $\partial$ is a locally nilpotent
higher derivation.  This iterative higher
derivation is called the canonical higher derivation associated to $\delta$.
In this case, we have a map ${\rm Der}(A)\to {\rm Der}^H(A)$
sending $\delta$ to $(\partial_i)$ as defined by \eqref{E6.1.2}. Hence
the map ${\rm Der}(A)\to {\rm Der}^H(A)$ is the right inverse of
the map ${\rm Der}^H(A)\to {\rm Der}(A)$.

The following lemmas are easy.

\begin{lemma}\cite[Lemma 2.2]{BeZ}
\label{xxlem6.2} Let $\partial:=(\partial_i)_{i\geq 0}$ be a
 higher derivation.
\begin{enumerate}
\item[(1)]
Suppose $\partial$ is locally nilpotent.
For any $c\in k$, $G_{c\partial}: A\to A$ defined by
\begin{equation}
\label{E6.2.1}\tag{E6.2.1}
G_{c\partial}: a\to \sum_{i\geq 0} c^i \partial_i(a)
\end{equation}
is an algebra automorphism of $A$.
\item[(2)]
If $\partial$ is iterative and satisfies Definition {\rm{\ref{xxdef6.1}(3a)}},
then $G_{\partial, t}$ as defined in \eqref{E6.1.1}
is an algebra  automorphism of $A$.
As a consequence, $\partial$ is locally nilpotent.
\item[(3)]
If $G: A[t]\to A[t]$ be a $k[t]$-algebra automorphism and
if $G(a)\equiv a \mod t$ for all $a\in A$, then $G=G_{\partial, t}$
for some $\partial\in \LND^H(A)$.
\end{enumerate}
\end{lemma}

\begin{lemma}
\label{xxlem6.3}
Let $A$ be a connected graded algebra. Let $g$ be a unipotent
automorphism of $A$. For every homogeneous element $a\in A$,
define $\partial_i(a)$ to be the $(i+\deg a)$-degree piece of
$g(a)$. Then $\partial:=(\partial_i)$ is a locally nilpotent
higher derivation and $g=G_{1\partial}$ as defined in \eqref{E6.2.1}.
As a consequence, if $\LND^H(A)=0$, then $\Aut_{uni}(A)=\{Id_A\}$.
\end{lemma}

We now recall the definition of the Makar-Limanov invariant.

\begin{definition}
\label{xxdef6.4}
Let $A$ be an algebra over $k$. Let $*$ be either blank or $^H$.
\begin{enumerate}
\item[(1)]
The {\it Makar-Limanov$^*$ invariant} \cite{Ma} of $A$ is defined to be
\begin{equation}
\label{E6.4.1}\tag{E6.4.1}
{\rm ML}^*(A) \ = \ \bigcap_{\delta\in {\rm LND}^*(A)} {\rm ker}(\delta).
\end{equation}
This means that we have original $\ML(A)$, as well as, $\ML^H(A)$.
\item[(2)]
We say that $A$ is \emph{$\LND^*$-rigid} if ${\rm ML}^*(A)=A$,
or $\LND^*(A)=\{0\}$.
\item[(3)]
We say that $A$ is \emph{strong $\LND^*$-rigid} if ${\rm ML}^*(A[t_1,\cdots,t_d])=A$,
for all $d\geq 1$.
\end{enumerate}
\end{definition}

By \cite[Theorem 3.6]{BeZ}, if $A$ is a finitely generated domain of 
finite GK-dimension and  $\ML^H(A)=A$, then $A$ is cancellative. Now 
we prove Theorem \ref{xxthm0.8}.

\begin{theorem}
\label{xxthm6.5}
Let $A$ be an algebra with Nakayama automorphism $\mu$. Suppose
that  $A^\times$ is in the center of $A$. Then $\mu$ commutes with
every locally nilpotent higher derivation $\partial=(\partial_i)$ of $A$,
that is, $\mu \partial_i = \partial_i \mu$ for all $i$.
\end{theorem}

\begin{proof} Let $G_{t,\partial}$ be the automorphism of $A[t]$
defined as in \eqref{E6.1.1}. Note that $(A[t])^\times
=A^\times$ is in the center of $A[t]$. By Theorem \ref{xxthm4.2},
$\mu_{A[t]}$ commutes with $G_{t,\partial}$. Since
$\mu_{A[t]}=\mu_A\otimes Id_{k[t]}$, $\mu_A$
commutes with $\partial_i$ for all $i$.
\end{proof}

The next lemma is similar to Lemma \ref{xxlem4.3}.

\begin{lemma}
\label{xxlem6.6} Let $A$ be an algebra with Nakayama automorphism
$\mu_A$ and such that $A^\times =k^\times$.
Suppose that $\{x_1,\cdots,x_n\}$ is a set of generators of $A$ such that
$\mu_A(x_i)=\lambda_i x_i$ for all $i$ and that the
set of the ordered monomials
$\{x_1^{a_1}\cdots x_n^{a_n}\mid a_1,\cdots,a_n\geq 0\}$
spans the whole algebra $A$.
Assume that $\lambda_1$ cannot be written as $\prod_{j>1} \lambda_j^{b_i}$
for any $b_j\geq 0$. Then $x_1\in {\rm ML}^H(A[y_1,\cdots,y_w])$.
\end{lemma}

\begin{proof} Consider the algebra $B:=A[y_1,\cdots, y_w]$ and $C=B[t]$.
Then $B$ and $C$ satisfy the hypotheses of the lemma. Let $\partial
=(\partial_i)$ be a locally nilpotent higher derivation of  $B$ and consider
the algebra automorphism $G_{t,\partial}: B[t]\to B[t]$. Since $C=B[t]$
satisfies the hypothesis, by Lemma \ref{xxlem4.3},
$G_{t,\partial}(x_1)= cx_1$
for some $c\in k$. But $G_{t,\partial}(x_1)=x_1+
\sum_{i=1}^{\infty} \partial_i (x_1) t^i$.
Thus $c=1$ and $x_1\in \ker \partial$. The assertion follows.
\end{proof}

\begin{definition} \cite[Definition 1.1]{BeZ}
\label{yydef6.7}
Let $A$ be an algebra.
\begin{enumerate}
\item
We call $A$ {\it cancellative} if $A[y]\cong B[y]$ for
some algebra $B$ implies that $A\cong B$.
\item
We call $A$ {\it strongly cancellative} if, for any $d\geq 1$,
$$A[y_1,\cdots,y_d]\cong B[y_1,\cdots,y_d]$$ for some algebra
$B$ implies that $A\cong B$.
\end{enumerate}
\end{definition}

Now we prove the second half of Corollary \ref{xxcor0.7}.

\begin{corollary}
\label{xxcor6.8}
If $A$ is one of the following algebras, then
$\ML^H(A[y_1,\cdots, y_n])=A$. As a consequence,
$A$ is also strongly cancellative.
\begin{enumerate}
\item[(1)]
$A(1)$ where $p_{ij}$ are generic.
\item[(2)]
$A(2)$
where $p$ is not a root of unity.
\item[(3)]
$A(6)$ where $\beta$ is not a root of unity.
\item[(4)]
$A(7)$ where $p$ is not a root of unity.
\end{enumerate}
\end{corollary}

\begin{proof} (1) By the proof of Proposition \ref{xxpro4.4}
for $A=A(1)$ with generic $p_{ij}$,
the hypothesis of Lemma \ref{xxlem6.6} was
checked for $\{x_1,x_2,x_3\}=\{t_1,t_2,t_3\}$.
So $t_1\in \ML^H(A[y_1,\cdots, y_n])$. By symmetry,
$t_1,t_2\in  \ML^H(A[y_1,\cdots, y_n])$. Hence
$$A\subset \ML^H(A[y_1,\cdots, y_n]).$$
It is obvious that
$$A\supset  \ML^H(A[y_1,\cdots, y_n]).$$
Therefore the assertion follows.

(2) By the proof of Proposition \ref{xxpro4.5}
for $A=A(2)$ with $p$ not a root of unity,
the hypothesis of Lemma \ref{xxlem6.6} was
checked for $\{x_1,x_2,x_3\}=\{t_2,t_1,t_3\}$.
So $t_2\in \ML^H(A[y_1,\cdots, y_n])$.
Similarly, $t_3\in \ML^H(A[y_1,\cdots, y_n])$.
By the relation $t_1^2=t_3t_2-pt_2t_3$, we have
$t_1^2\in \ML^H(A[y_1,\cdots, y_n])$. A basic
property of locally nilpotent derivation implies
that $t_1\in \ML^H(A[y_1,\cdots, y_n])$.
The assertion follows.

(3,4) The assertion follows by
using a similar idea as above and the fact
proved in Proposition \ref{xxpro4.6}.
\end{proof}

Now we are ready to prove Corollary \ref{xxcor0.9}.

\begin{proof}[Proof of Corollary \ref{xxcor0.9}]
For algebras $A(1), A(2), A(6)$ and $A(7)$, see
Corollary \ref{xxcor6.8}. For algebras $A(3), A(4)$
and $A(5)$, we have seen that the center is trivial
[Lemma \ref{xxlem4.10}(3) and Proposition
\ref{xxpro4.11}(3,5)]. By \cite[Proposition 1.3]{BeZ},
these algebras are cancellative.
\end{proof}

We use the following diagram to illustrate the ideas involved.
\begin{small}
$$\begin{CD}
\mu_A @. @. @. {\text{ZCP}}  \\
@VV {\text{controls}}V @. @. @A {\text{answers}} AA\\
\Aut(A)@< \supseteq << \Aut_{uni}(A) @>>> \LND^H(A[y_1,\cdots,y_n]) @>>>
\ML^H(A[y_1,\cdots,y_n])
\end{CD}
$$
\end{small}
In the paper \cite{BeZ}, the authors use similar ideas with $\mu_A$
being replaced by the discriminant.

\section*{Acknowledgments}
This paper is a continuation of the project started in \cite{CWZ, CKWZ}.
The authors would like to thank Yan-Hong Bao, Kenneth Chan, Ji-Wei He, 
Chelsea Walton, Ellen Kirkman and Jim Kuzmanovich for many useful 
conversations on the subject. Some ideas  were generated by reading 
examples constructed by Ellen Kirkman and Jim Kuzmanovich.
The authors also thank the referee for his/her careful reading and 
valuable comments.
J.-F. L\"u was supported by  NSFC  (Grant No.11001245). X.-F. Mao was 
supported by NSFC (Grant No.11001056), the Key Disciplines of 
Shanghai Municipality (Grant No.S30104) and the Innovation Program of
Shanghai Municipal Education Commission (Grant No.12YZ031).
J.J. Zhang was supported by the US National Science
Foundation (NSF grant No. DMS 0855743 and  DMS 1402863).

\providecommand{\bysame}{\leavevmode\hbox to3em{\hrulefill}\thinspace}
\providecommand{\MR}{\relax\ifhmode\unskip\space\fi MR }
\providecommand{\MRhref}[2]{%
\href{http://www.ams.org/mathscinet-getitem?mr=#1}{#2} }
\providecommand{\href}[2]{#2}

\end{document}